\documentclass{amsart}
\usepackage{shortsalch, xy, eufrak}
\usepackage{breqn}
\xyoption{all}
\title{Height four formal groups with quadratic complex multiplication.}
\date{July 2016.}
\begin{document}
\begin{abstract}
We construct spectral sequences for computing the cohomology of automorphism groups of formal groups with complex multiplication by a $p$-adic number ring. We then compute the cohomology of the group of automorphisms of a height four formal group law which commute with complex multiplication by the ring of integers in the field $\mathbb{Q}_p(\sqrt{p})$, for primes $p>5$. This is a large subgroup of the height four strict Morava stabilizer group. The group cohomology of this group of automorphisms turns out to have cohomological dimension $8$ and total rank $80$. We then run the $K(4)$-local $E_4$-Adams spectral sequence to compute the homotopy groups of the homotopy fixed-point spectrum of this group's action on the Lubin-Tate/Morava spectrum $E_4$.
\end{abstract}
\maketitle
\tableofcontents

\section{Introduction.}

This paper is intended as a companion and sequel to~\cite{formalmodules4}.
In J.P. May's thesis~\cite{MR2614527}, he used the results of Milnor and Moore's paper~\cite{MR0174052} to set up spectral sequences for computing the cohomology of the Steenrod algebra, the input for the classical Adams spectral sequence; in chapter~6 of~\cite{MR860042}, Ravenel adapts May's spectral sequences for the purpose of computing the cohomology of automorphism groups of formal group laws, i.e., Morava stabilizer groups, which are the input for various spectral sequence methods for computing stable homotopy groups of spheres and Bousfield localizations of various spectra.
In this paper we adapt Ravenel's tools to the task of computing the cohomology of automorphism groups of formal group laws with complex multiplication by a $p$-adic number ring $A$, i.e., formal $A$-modules. We show (Theorem~\ref{its a closed subgroup}) that these automorphism groups are closed subgroups of the Morava stabilizer groups, so that the machinery of~\cite{MR2030586} can be used to construct and compute the homotopy fixed-point spectra of the action of these automorphism groups on Lubin-Tate/Morava $E$-theory spectra; and then, most importantly, we actually use all this machinery to do some nontrivial computations: in Theorem~\ref{coh of ht 2 fm}, we compute the cohomology of the group $\strictAut({}_1\mathbb{G}_{1/2}^{\hat{\mathbb{Z}}_p\left[\sqrt{p}\right]})$ of strict automorphisms of a height four formal group law which commute with complex multiplication by the ring of integers in the field $\mathbb{Q}_p(\sqrt{p})$, for primes $p>5$. This is a large subgroup of the height four strict Morava stabilizer group. The group cohomology $H^*(\strictAut({}_1\mathbb{G}_{1/2}^{\hat{\mathbb{Z}}_p\left[\sqrt{p}\right]}); \mathbb{F}_p)$ turns out to have cohomological dimension $8$, total rank $80$, and Poincar\'{e} series
\begin{dmath*} (1+s)^4(1+3s^2+s^4) =  1+ 4s + 9s^2 + 16s^3 + 20s^4 + 16s^5 + 9s^6 + 4s^7 + s^8,\end{dmath*}
i.e., $H^1(\strictAut({}_1\mathbb{G}_{1/2}^{\hat{\mathbb{Z}}_p\left[\sqrt{p}\right]}); \mathbb{F}_p)$ is a four-dimensional $\mathbb{F}_p$-vector space,
$H^2(\strictAut({}_1\mathbb{G}_{1/2}^{\hat{\mathbb{Z}}_p\left[\sqrt{p}\right]}); \mathbb{F}_p)$ is a nine-dimensional $\mathbb{F}_p$-vector space, and so on.

We then run the descent/$K(4)$-local $E_4$-Adams spectral sequence to compute the homotopy groups of the homotopy fixed-point spectrum
$E_4^{h\Aut({}_{1}\mathbb{G}_{1/2}^{\hat{\mathbb{Z}}_p[\sqrt{p}]}\otimes_{\mathbb{F}_p}\overline{\mathbb{F}}_p)\rtimes \Gal(\overline{k}/\mathbb{F}_p)}$
smashed with the Smith-Toda complex $V(3)$ (which exists, since we are still assuming that $p>5$). The computation is Theorem~\ref{topological computation}; the resulting $V(3)$-homotopy groups have Poincar\'{e} series
\begin{dmath*} \left( s^{-6}+ 2s^{-3} + 1 + s^{2p-7} + s^{2p-6} + s^{2p-4}  + s^{2p-3} + s^{4p-8}  + 2s^{4p-7}+ s^{4p-6} +  s^{2p^2 - 2p-5 } + s^{2p^2-2p-4} + s^{2p^2-2p-1}+ s^{4p^2-4p-4}  +  2s^{4p^2-4p-3}+ s^{4p^2-4p-2} \right)(1+s^{-1})^2 \sum_{n=-\infty}^{\infty} s^{2(p^2-2)n}.\end{dmath*}

In work currently in preparation, we use some of the computations in this paper as input for further, more difficult computations which eventually arrive at the cohomology of the height four Morava stabilizer group at primes $p>7$; see~\cite{height4} for some details.

The computations in~\cref{computations section} of this paper appeared already in the (unpublished, and not submitted for publication) announcement~\cite{height4}; any version of that announcement which is ever submitted for journal publication will feature at most only an abbreviated version of these computations, with the idea that the complete versions are those provided in the present paper.

\begin{conventions}
\begin{itemize}
\item In this paper, all formal groups and formal modules are implicitly assumed to be {\em one-dimensional.}
\item Throughout, we will use Hazewinkel's generators for $BP_*$ (and, more generally, for the classifying ring $V^A$ of $A$-typical formal $A$-modules, where $A$ is a discrete valuation ring).
\item By a ``$p$-adic number field'' we mean a finite field extension of the $p$-adic rationals $\mathbb{Q}_p$ for some prime $p$.
\item When a ground field $k$ is understood from context, we will write $\Lambda( x_1, \dots ,x_n)$ for the exterior/Grassmann $k$-algebra with generators $x_1, \dots ,x_n$.
\item Given a field $k$, we write $k\{ x_1, \dots ,x_n\}$ for the abelian Lie algebra over $k$ with basis $x_1, \dots ,x_n$.
\item When $L$ is a restricted Lie algebra over a field $k$ and $M$ is a module over the restricted enveloping algebra $VL$, we write $H^*_{res}(L, M)$ for restricted Lie algebra cohomology, i.e., $H^*_{res}(L, M) \cong \Ext_{VL}^*(k, M)$, and we write $H^*_{unr}(L, M)$ for unrestricted Lie algebra cohomology, i.e., $H^*_{unr}(L, M) \cong \Ext_{UL}^*(k, M)$, where $UL$ is the universal enveloping algebra of $L$.
\item Whenever convenient, we make use, without comment, of the well-known theorem of Milnor and Moore, from~\cite{MR0174052}: given a field $k$ of characteristic $p>0$, the functors $P$ (restricted Lie algebra of primitives) and $V$ (restricted enveloping algebra) establish an equivalence of categories between restricted Lie algebras over $k$ and primitively generated cocommutative Hopf algebras over $k$, and this equivalence preserves cohomology, i.e., $H^*_{res}(L, M) \cong \Ext_{VL}^*(k, M)$. We write $H^*_{res}(L, M)$ for restricted Lie algebra cohomology, and we write $H^*_{unr}(L, M)$ for unrestricted Lie algebra cohomology, i.e., $H^*_{unr}(L, M) \cong \Ext_{UL}^*(k, M)$, where $UL$ is the universal enveloping algebra of $L$.
\item Whenever convenient, we make use of the Chevalley-Eilenberg complex of a Lie algebra $L$ to compute (unrestricted) Lie algebra cohomology $H^*_{unr}(L, M)$, as in~\cite{MR0024908}.
\item Many of the differential graded algebras in this paper have a natural action by a finite cyclic group; given an action by a finite cyclic group $C_n$ on some DGA, will always fix a generator for $C_n$ and write $\sigma$ for that generator.
\end{itemize}
\end{conventions}

\section{Review of Ravenel's filtration and associated May spectral sequences.}

This paper continues from~\cite{formalmodules4}; for a brief introduction to formal $A$-modules, their moduli, $A$-typicality, $A$-height, and so on, the reader can consult that paper. A more complete account is in~\cite{MR745362}, and an even more complete account is chapter~21 of~\cite{MR2987372}.
Briefly, the most important fact we will use is that, for $A$ the ring of integers in a $p$-adic number field, the classifying ring of $A$-typical formal $A$-modules is $V^A \cong A[v_1^A, v_2^A, \dots]$, and the classifying ring of strict isomorphisms of $A$-typical formal $A$-modules is $V^AT \cong V^A[t_1^A, t_2^A, \dots]$,
with $v_n^A$ and $t_n^A$ each in grading degree $2(q^n-1)$, where $q$ is the cardinality of the residue field of $A$.

Definition~\ref{G notation} and Theorem~\ref{map induced in S(n)} appeared in~\cite{formalmodules4}:
\begin{definition}\label{G notation}
Let $K$ be a $p$-adic number field with ring of integers $A$ and residue field $k$, and let $n$ be a positive integer.
Let $k^{\prime}$ be a field extension of $k$, and let $\alpha\in (k^{\prime})^{\times}$.
We write ${}_{\alpha}\mathbb{G}^A_{1/n}$ for the the formal $A$-module over $k^{\prime}$ classified by the map $V^A \rightarrow k^{\prime}$ sending $v_n^A$ to $\alpha$ and sending $v_i^A$ to zero if $i\neq n$.
\end{definition}

\begin{theorem}\label{map induced in S(n)}
Let $L/K$ be a finite field extension of degree $d$, with $K,L$ $p$-adic number fields with rings of integers $A,B$ respectively. Let $k,\ell$ be the residue fields of $A$ and $B$,
let $e$ be the ramification degree and $f$ the residue degree of $L/K$, let $q$ be the cardinality of $\ell$, and let $\pi_A,\pi_B$ be uniformizers for $A,B$, respectively. Let $n$ be a positive integer.
If $\ell^{\prime}$ is a field extension of $\ell$ and $\beta\in (\ell^{\prime})^{\times}$, then the underlying formal $A$-module of ${}_{\beta}\mathbb{G}_{1/n}^B$ is ${}_{\alpha }\mathbb{G}_{1/dn}^A$, where 
\[ \alpha = \frac{\pi_A}{\pi_B^e}\beta^{\frac{q^{en}-1}{q^n-1}}.\]

Furthermore, the ring map
\begin{align} \nonumber k(\alpha)[\strictAut({}_{\alpha}\mathbb{G}_{1/dn}^A)]^* = k(\alpha)[t_1^A,t_2^A,\dots]/( t_i^A \alpha^{q^{ei}-1} - (t_{i}^A)^{q^{en}} \ \ \ \forall i ) \\
\label{map 4308443}\rightarrow 
 \ell[t_1^B,t_2^B,\dots]/( t_i^B \beta^{q^i-1}-(t_i^B)^{q^n} \ \ \ \forall i )
= \ell[\strictAut({}_{\beta}\mathbb{G}_{1/n}^B)]^* \end{align}
classifying the strict formal $A$-module automorphism of ${}_{\alpha}\mathbb{G}_{1/dn}^A$ underlying the universal strict formal $B$-automorphism of ${}_{\beta}\mathbb{G}_{1/n}^B$ sends $t_i^A$ to $t_{i/f}^A$ if $i$ is divisible by the residue degree $f$ of $L/K$, and sends $t_i^A$ to zero if $i$ is not divisible by the residue degree $f$.
\end{theorem}

Ravenel writes $S(n)$ for the Hopf algebra $\mathbb{F}_{p^n}[\strictAut({}_1\mathbb{G}_{1/n}^{\hat{\mathbb{Z}}_p})]^*$; most of his constructions and computations work equally well for $\mathbb{F}_{p}[\strictAut({}_1\mathbb{G}_{1/n}^{\hat{\mathbb{Z}}_p})]^*$.

Definition~\ref{ravenel numbers} and Theorems~\ref{assoc graded} and~\ref{ravenels spectral sequences} were given in section~6.3 of~\cite{MR860042}.
\begin{definition}{\bf (Ravenel's numbers.)} \label{ravenel numbers} Fix a prime number $p$ and a positive integer $n$. Let $d_{n,i}$ be the integer defined by the formula
\[ d_{n,i} = \left\{ \begin{array}{ll} 0 & \mbox{\ if\ } i\leq 0 \\ \max\{ i,pd_{n,i-n}\} & \mbox{\ if\ } i>0.\end{array} \right. \]
(Clearly $d_{n,i}$ depends on the prime number $p$, but the choice of prime $p$ is suppressed from the notation for $d_{n,i}$.)

Now equip the continuous $\mathbb{F}_p$-linear dual Hopf algebra
\[ \mathbb{F}_p[\strictAut({}_1\mathbb{G}_{1/n}^{\hat{\mathbb{Z}}_p})]^*
\cong \mathbb{F}_p\otimes_{BP_*}BP_*BP\otimes_{BP_*} \mathbb{F}_p
\cong \mathbb{F}_p[t_1, t_2, \dots ]/(t_i^{p^n}-t_i\ \ \  \forall i) \]
with the increasing filtration in which the element $t_i^{j}$ is in filtration degree $s_p(j)d_{n,i}$, where $s_p(j)$ is the sum of the digits in the base $p$ expansion of $j$. Here the $BP_*$-module structure of $\mathbb{F}_p$ is given by the ring map $BP_*\rightarrow \mathbb{F}_p$ sending $v_n$ to $1$ and sending $v_i$ to $0$ for all $i\neq 0$. We call this filtration the {\em Ravenel filtration.}
\end{definition}

\begin{theorem}\label{assoc graded}
The Ravenel filtration is an increasing Hopf algebra filtration, and its 
associated graded Hopf algebra $E^0S(n)$ is $\mathbb{F}_p$-linearly dual to a primitively generated Hopf algebra.
The Hopf algebra $E^0S(n)$ is isomorphic, as an $\mathbb{F}_p$-algebra, to a truncated polynomial algebra:
\begin{eqnarray*} E^0\mathbb{F}_p[\strictAut({}_1\mathbb{G}_{1/n}^{\hat{\mathbb{Z}}_p})]^* & \cong & 
\mathbb{F}_{p}[t_{i,j} : 1\leq i, j\in \mathbb{Z}/n\mathbb{Z}]/t_{i,j}^p,\end{eqnarray*}
where the coproduct is given by
\begin{eqnarray}\label{coproduct formula in E^0S(n)}
\Delta(t_{i,j}) = \left\{ \begin{array}{ll} \sum_{0\leq k\leq i} t_{k,j}\otimes t_{i-k,k+j} & \mbox{if\ } i<\frac{pn}{p-1}, \\
 \sum_{0\leq k\leq i} t_{k,j}\otimes t_{i-k,k+j} + \overline{b}_{i-n,j+n-1} & \mbox{if\ } i=\frac{pn}{p-1},\\
 t_{i,j}\otimes 1 + 1\otimes t_{i,j} + \overline{b}_{i-n,j+n-1} & \mbox{if\ } i>\frac{pn}{p-1},\end{array}\right. ,
\end{eqnarray}
where $t_{i,j}$ is the element of $E^0\mathbb{F}_p[\strictAut({}_1\mathbb{G}_{1/n}^{\hat{\mathbb{Z}}_p})]^*$ corresponding to $t_i^{p^j}\in \mathbb{F}_p[\strictAut({}_1\mathbb{G}_{1/n}^{\hat{\mathbb{Z}}_p})]^*$,
$t_{0,j} = 1$, and $\overline{x}$ is the image in $E^0\mathbb{F}_p[\strictAut({}_1\mathbb{G}_{1/n}^{\hat{\mathbb{Z}}_p})]^*$ of an element $x\in \mathbb{F}_p[\strictAut({}_1\mathbb{G}_{1/n}^{\hat{\mathbb{Z}}_p})]^*$. The $\overline{b}$ elements have a fairly complicated combinatorial description; 
see~4.3.14 of~\cite{MR860042}.

The Hopf algebra $E^0\mathbb{F}_p[\strictAut({}_1\mathbb{G}_{1/n}^{\hat{\mathbb{Z}}_p})]^*$ is the $\mathbb{F}_{p}$-linear dual of the restricted enveloping algebra of a
restricted Lie algebra $L(n)$. Let $x_{i,j}$ denote the $\mathbb{F}_{p}$-linear functional on $E^0\mathbb{F}_p[\strictAut({}_1\mathbb{G}_{1/n}^{\hat{\mathbb{Z}}_p})]^*$ which is 
dual to $t_{i,j}$; then the set
\[ \left\{ x_{i,j}: i>0, j\in \mathbb{Z}/n\mathbb{Z}\right\}\]
is a $\mathbb{F}_{p^n}$-linear basis for $L(n)$. We describe the bracket and the restriction $\xi$ on $L(n)$:
\begin{eqnarray*}  [ x_{i,j},x_{k,l}] & = & \left\{ \begin{array}{ll} 
\delta^l_{i+j}x_{i+k,j} - \delta^j_{k+l}x_{i+k,l} & \mbox{if\ } i+k\leq \frac{pn}{p-1}, \\
0 & \mbox{if\ } i+k >\frac{pn}{p-1} .\end{array}\right. ,\\
\xi(x_{i,j}) & = & \left\{\begin{array}{ll}
x_{i+n,j+1} & \mbox{if\ } i>\frac{n}{p-1}\mbox{\ or\ } i=\frac{n}{p-1} \mbox{\ and\ } p>2,\\
x_{2n,j} + x_{2n,j+1} & \mbox{if\ } i=n\mbox{\ and\ } p=2,\\
0 & \mbox{if\ } i<\frac{n}{p-1}.\end{array}\right. ,\end{eqnarray*}
where $\delta^a_b = 1$ if $a\equiv b$ modulo $n$, and $\delta^a_b= 0$ if $a\nequiv b$ modulo $n$.
\end{theorem}

The two spectral sequences of J.P. May's thesis~\cite{MR2614527}, in this context, take the form:
\begin{theorem}\label{ravenels spectral sequences}
We have spectral sequences
\begin{align} 
\label{may 1 ss} E_2^{s,t,u} \cong H^{s,t}_{unr}(L(n); \mathbb{F}_p)\otimes_{\mathbb{F}_p} P( b_{i,j}: i\geq 1, j \in\mathbb{Z}/n\mathbb{Z})  &\Rightarrow H^{s,t}_{res}(L(n); \mathbb{F}_p) \\
\nonumber d_r: E_r^{s,t,u} &\rightarrow E_r^{s+1,t,u+r-1} \\
\label{may 2 ss} E_1^{s,t} \cong H^{s,t}_{res}(L(n); \mathbb{F}_p) &\Rightarrow H^{s,t}(\strictAut({}_1\mathbb{G}_{1/n}^{\hat{\mathbb{Z}}_p}); \mathbb{F}_p) \\
\nonumber d_r: E_r^{s,t} &\rightarrow E_r^{s+1,t-r}  ,\end{align}
where $H^*_{unr}$ is (unrestricted) Lie algebra cohomology and $H^*_{res}$ is restricted Lie algebra cohomology. 
Furthermore, the filtered DGA which gives rise to spectral sequence~\ref{may 1 ss} splits as a tensor product of a term with trivial $E_{\infty}$-term with a term whose $E_2$-term is 
\[ H^*\left(L(n,\floor{\frac{pn}{p-1}}); \mathbb{F}_p\right) \otimes_{\mathbb{F}_p} P\left(b_{i,j}: 1\leq i\leq \frac{n}{p-1}, j\in\mathbb{Z}/n\mathbb{Z}\right),\]
where $L(n,\floor{\frac{pn}{p-1}})$ is the quotient restricted Lie algebra of $L(n)$ in which we quotient out by the elements $x_{i,j}$ with $i> \floor{\frac{pn}{p-1}}$.
Consequently there exists a spectral sequence
\begin{align}
\label{may 1 ss reduced} E_2^{s,t,u} \cong H^{s,t}_{unr}(L(n,\floor{\frac{pn}{p-1}}); \mathbb{F}_p)\otimes_{\mathbb{F}_p} P( b_{i,j}: 1\leq i\leq \frac{n}{p-1}, 0\leq j < n)  &\Rightarrow H^{s,t}_{res}(L(n); \mathbb{F}_p) \\
\nonumber d_r: E_r^{s,t,u} &\rightarrow E_r^{s+1,t,u+r-1} \end{align}
\end{theorem}

Computation of the Chevalley-Eilenberg complex of $L(n)$ and of $L(n,\floor{\frac{pn}{p-1}})$ is routine, and appears in Theorem~6.3.8 of~\cite{MR860042}:
\begin{theorem}\label{ravenels chevalley-eilenberg computation}
Let $\mathcal{K}(n,m)$ be the differential graded $\mathbb{F}_p$-algebra which is the exterior algebra $\Lambda( h_{i,j}: 1\leq i \leq m, j\in\mathbb{Z}/n\mathbb{Z})$ with differential \[ d(h_{i,j}) = \sum_{k=1}^{i-1} h_{k,j} h_{i-k,j+k},\]
with the convention that $h_{i,k+n} = h_{i,k}$.
Then $H^*(\mathcal{K}(n,\floor{\frac{pn}{p-1}})) \cong H^*_{unr}(L(n, \floor{\frac{pn}{p-1}}); \mathbb{F}_p)$.
\end{theorem}

\section{Generalizations for formal $A$-modules.}

Recall that a graded Hopf algebra $A$ over a field $k$ is said to be {\em finite-type} if, for each $n\in\mathbb{Z}$, the grading degree $n$ summand $A^n$ of $A$ is a finite-dimensional $k$-vector space.
\begin{prop}\label{assoc graded with complex mult}
Let $K/\mathbb{Q}_p$ be a field extension of degree $d$ and ramification degree $e$ and residue degree $f$. Let $A$ be the ring of integers of $K$, let $\pi$ be a uniformizer for $A$, and let $k$ be the residue field of $A$. 
Let $n$ be a positive integer, and let $\omega \in \overline{k}$ be a 
$\frac{q^{en}-1}{q^n-1}$th root of $\frac{\pi^e}{p}$. 
Then the underlying formal $\hat{\mathbb{Z}}_p$-module of ${}_{\omega}\mathbb{G}_{1/n}^A$ is ${}_{1}\mathbb{G}_{1/dn}^{\hat{\mathbb{Z}}_p}$, and 
the Ravenel filtration on $\mathbb{F}_p[\strictAut({}_1\mathbb{G}_{1/dn}^{\hat{\mathbb{Z}}_p})]^*$ induces a compatible filtration on the Hopf algebra
\begin{equation}\label{iso 4309853} k(\omega)[\strictAut({}_{\omega}\mathbb{G}_{1/n}^{A})]^* \cong
 k(\omega)[t_f, t_{2f}, \dots ]/(t_{if}^{q^n} - \omega^{q^i-1}t_{if}\ \ \  \forall i ).\end{equation}
The associated graded Hopf algebra 
$E^0k(\omega)[\strictAut({}_{\omega}\mathbb{G}_{1/n}^{A})]^*$
is the graded $k$-linear dual of a primitively generated finite-type Hopf algebra, which, as a quotient of 
$E^0\mathbb{F}_p[\strictAut({}_1\mathbb{G}_{1/dn}^{\hat{\mathbb{Z}}_p})]^*\otimes_{\mathbb{F}_p} k(\omega)$, is given by
\[ k(\omega)[t_{if,j} : i\geq 1, j\in\mathbb{Z}/fn\mathbb{Z}]/(t_{i,j}^p\ \ \ \forall i,j),\]
with coproduct
\begin{eqnarray}\label{coproduct on associated graded}
\Delta(t_{i,j}) = \left\{ \begin{array}{ll} \sum_{0\leq k\leq i} t_{k,j}\otimes t_{i-k,k+j} & \mbox{if\ } i<\frac{pdn}{p-1}, \\
 \sum_{0\leq k\leq i} t_{k,j}\otimes t_{i-k,k+j} + \overline{b}_{i-dn,j+dn-1} & \mbox{if\ } i=\frac{pdn}{p-1},\\
 t_{i,j}\otimes 1 + 1\otimes t_{i,j} + \overline{b}_{i-dn,j+dn-1} & \mbox{if\ } i>\frac{pdn}{p-1}.\end{array}\right. 
\end{eqnarray}
Here $t_{i,j}$ is the element of $E^0k(\omega)[\strictAut({}_{\omega}\mathbb{G}_{1/n}^{A})]^*$ corresponding to $t_i^{p^j}\in k(\omega)[\strictAut({}_{\omega}\mathbb{G}_{1/n}^{A})]^*$, and:
\begin{eqnarray}
\nonumber t_{0,j} & = & 1, \\
\nonumber t_{i,j} & = & 0 \mbox{\ if\ } f\nmid i>0, \\
\nonumber t_{i,j+fn} & = & \omega^{p^j(q^i-1)}t_{i,j},\mbox{\ and} \\
\label{b formula 1} \overline{b}_{i,j+fn} & = & \omega^{p^{j+1}(q^i-1)}\overline{b}_{i,j}.\end{eqnarray}
\end{prop}
\begin{proof}
The claim that the underlying formal $\hat{\mathbb{Z}}_p$-module of ${}_{\omega}\mathbb{G}_{1/n}^A$ is ${}_{1}\mathbb{G}_{1/dn}^{\hat{\mathbb{Z}}_p}$ is simply a special case of Theorem~\ref{map induced in S(n)}, as is the isomorphism~\ref{iso 4309853}.
The fact that the Ravenel filtration on 
$\mathbb{F}_p[\strictAut({}_1\mathbb{G}_{1/dn}^{\hat{\mathbb{Z}}_p})]^*$ induces a filtration on the Hopf algebra $k(\omega)[\strictAut({}_{\omega}\mathbb{G}_{1/n}^{A})]^*$ is straightforward: the 
Hopf algebra $k(\omega)[\strictAut({}_{\omega}\mathbb{G}_{1/n}^{A})]^*$
is the quotient of 
$\left(\mathbb{F}_p[\strictAut({}_1\mathbb{G}_{1/dn}^{\hat{\mathbb{Z}}_p})]^*\right)\otimes_{\mathbb{F}_p} k(\omega)$ by the ideal $I$ generated by $t_{if}^{q^n} - \omega^{q^{if}-1}t_{if}$ for all $i$ and by $t_j$ for all $j$ not divisible by $f$; these generators for this ideal $I$ are all homogeneous in the Ravenel filtration. 

Computation of the associated graded, including the formula for the coproduct, is routine: simply reduce the formulas of Theorem~\ref{assoc graded} modulo $I$.
Deriving formula~\ref{b formula 1} requires consulting the definition of $b_{i,j}$ in 4.3.14 of~\cite{MR860042} in terms of Witt polynomials; the essential observations here are that $b_{i,j+1}$, modulo $p$, is obtained from $b_{i,j}$ by replacing each element $t_m$ with $t_m^p$, and that $t_i^{ap^{j+fn}}\otimes t_i^{(p-a)p^{j+fn}} = 
(t_i^{p^{fn}})^{ap^{j}}\otimes (t_i^{p^{fn}})^{(p-a)p^{j}} = \omega^{p^{j+1}(p^{fi}-1)}t_i^{ap^j}\otimes t_i^{(p-a)p^j}$.

The fact that $E^0k(\omega)[\strictAut({}_{\omega}\mathbb{G}_{1/n}^{A})]^*$ is finite-type is immediate from its given presentation; and it is dual to a primitively generated Hopf algebra since its linear dual is a sub-Hopf-algebra of the linear dual of $E^0\mathbb{F}_p[\strictAut({}_1\mathbb{G}_{1/dn}^{\hat{\mathbb{Z}}_p})]^*\otimes_{\mathbb{F}_p} k(\omega)$, which is primitively generated.
\end{proof}

\begin{theorem} {\bf (Structure of $PE_0k(\omega)[\strictAut({}_{\omega}\mathbb{G}_{1/n}^{A})]$.)}\label{lie bracket computation}
 Let $K/\mathbb{Q}_p$ be a field extension of degree $d$ and ramification degree $e$ and residue degree $f$. Let $A$ be the ring of integers of $K$, let $\pi$ be a uniformizer for $A$, and let $k$ be the residue field of $A$. 
Let $n$ be a positive integer, and let $\omega \in \overline{k}$ be a 
$\frac{q^{en}-1}{q^n-1}$th root of $\frac{\pi^e}{p}$. 
Let $PE_0k(\omega)[\strictAut({}_{\omega}\mathbb{G}_{1/n}^{A})]$ be the
restricted Lie algebra of primitives in the $k(\omega)$-linear dual Hopf algebra $\left( E^0k(\omega)[\strictAut({}_{\omega}\mathbb{G}_{1/n}^{A})]^*\right)^*$.
Let $x^A_{i,j}$ be the element of
$PE_0k(\omega)[\strictAut({}_{\omega}\mathbb{G}_{1/n}^{A})]$ dual to the indecomposable $t_{i,j}\in E^0k(\omega)[\strictAut({}_{\omega}\mathbb{G}_{1/n}^{A})]^*$. Then 
$\{ x^A_{i,j} : f\mid i, j\in\mathbb{Z}/fn\mathbb{Z}\}$ is a $k(\omega)$-linear basis
for $PE_0k(\omega)[\strictAut({}_{\omega}\mathbb{G}_{1/n}^{A})]$, and dual to the natural surjection 
\[ \mathbb{F}_p[\strictAut({}_1\mathbb{G}_{1/dn}^{\hat{\mathbb{Z}}_p})]^*\otimes_{\mathbb{F}_p} k(\omega) \rightarrow 
 k(\omega)[\strictAut({}_{\omega}\mathbb{G}_{1/n}^{A})]^*,\]
we have an inclusion of restricted Lie algebras over $k(\omega)$:
\begin{eqnarray} \nonumber 
 PE_0k(\omega)[\strictAut({}_{\omega}\mathbb{G}_{1/n}^{A})]
  & \stackrel{\iota}{\longrightarrow} &
   PE_0k(\omega)[\strictAut({}_{1}\mathbb{G}_{1/dn}^{\hat{\mathbb{Z}}_p})] \\
\label{formula for iota} 
 \iota(x_{i,j}^A) & = & \sum_{\ell=0}^{e-1} 
  \omega^{p^{j}(q^i-1)\frac{q^{\ell n}-1}{q^n-1}} x_{i,j+\ell fn} .\end{eqnarray}
When $p = \pi^e$,
the bracket on $PE_0k(\omega)[\strictAut({}_{\omega}\mathbb{G}_{1/n}^{A})]$ is given by
\begin{equation} \label{bracket formula} [x_{i,j}^A,x_{k,\ell}^A] = \left\{ \begin{array}{lll} 
 \tilde{\delta}_{i+j}^{\ell} x_{i+k,j}^A - \tilde{\delta}_{k+\ell}^j x_{i+k,\ell}^A
  &  \mbox{\ if\ } & i+k\leq \frac{pdn}{p-1} \\
0 &  \mbox{\ if\ } & i+k > \frac{pdn}{p-1}\end{array} \right.
,\end{equation}
where $\tilde{\delta}_a^b=1$ if $a\equiv b$ modulo $fn$, and $\tilde{\delta}_a^b=0$ if $a\nequiv b$ modulo $fn$.
The restriction on $PE_0k(\omega)[\strictAut({}_{\omega}\mathbb{G}_{1/n}^{A})]$ is given by 
\begin{equation} \label{restriction formula} \xi(x_{i,j}^A) = \left\{ \begin{array}{lll} 
                      x_{i+n,j+1}^A & \mbox{\ if\ } & i > \frac{dn}{p-1} \\
                      x_{i+n,j+1}^A + x_{pi,j}^A & 
                                      \mbox{\ if\ } & i = \frac{dn}{p-1} 
                                      \mbox{\ and\ } fn\mid i \\
                      x_{i+n,j+1}^A & \mbox{\ if\ } & i = \frac{dn}{p-1} 
                                      \mbox{\ and\ } fn\nmid  i \\
                      x_{pi,j}^A &  
                                      \mbox{\ if\ } & i < \frac{dn}{p-1} 
                                      \mbox{\ and\ } fn\mid i \\
                      0 &             \mbox{\ if\ } & i < \frac{dn}{p-1} 
                                      \mbox{\ and\ } fn\nmid  i \\
\end{array}\right. ,\end{equation}
where $x_{i+n,fn}^A = x_{i+n,0}^A.$ 
\end{theorem}
\begin{proof}
Formula~\ref{formula for iota} follows from checking where elements in $E^0\mathbb{F}_p[\strictAut({}_1\mathbb{G}_{1/dn}^{\hat{\mathbb{Z}}_p})]^*$ are sent to 
$E^0k(\omega)[\strictAut({}_{\omega}\mathbb{G}_{1/n}^A)]^*$ by the canonical surjection 
$E^0\mathbb{F}_p[\strictAut({}_1\mathbb{G}_{1/dn}^{\hat{\mathbb{Z}}_p})]^*\otimes_{\mathbb{F}_{p^n}}k(\omega)\rightarrow E^0k(\omega)[\strictAut({}_{\omega}\mathbb{G}_{1/n}^A)]^*$; we use the description of the map~\ref{map 4308443} in Theorem~\ref{map induced in S(n)} to accomplish this. The map
$\mathbb{F}_p[\strictAut({}_1\mathbb{G}_{1/dn}^{\hat{\mathbb{Z}}_p})]^*\otimes_{\mathbb{F}_{p^n}}k(\omega)\rightarrow k(\omega)[\strictAut({}_{\omega}\mathbb{G}_{1/n}^A)]^*$ sends $t_i^{p^j}$ to 
\begin{align*}
 t_i^{p^{j_0+j_1fn}} 
  &= ((\dots ((t_i^{p^{fn}})^{p^{fn}})^{p^{fn}} \dots )^{p^{fn}})^{p^{j_0}} \\
  &= \omega^{p^{j_0}(q^i-1)\frac{q^{j_1n}-1}{q^n-1}}t_i^{p^{j_0}},\end{align*}
where
$j_0,j_1$ are the unique nonnegative integers such that $j=j_0+j_1fn$ and
$j_0<fn$; hence the map $E^0\mathbb{F}_p[\strictAut({}_1\mathbb{G}_{1/dn}^{\hat{\mathbb{Z}}_p})]^*\otimes_{\mathbb{F}_{p^n}}k(\omega)
\rightarrow E^0k(\omega)[\strictAut({}_{\omega}\mathbb{G}_{1/n}^A)]^*$ sends $t_{i,j}$ to
$\omega^{p^{j_0}(q^i-1)\frac{q^{j_1n}-1}{q^n-1}}t_{i,j_0}$.
Formula~\ref{formula for iota} follows at once.

Now suppose that $p = \pi^e$. We compute the Lie bracket in $PE_0k(\omega)[\strictAut({}_{\omega}\mathbb{G}_{1/n}^A)]$:
\begin{eqnarray*}
 [x_{i,j}^A,x_{k,\ell}^A] & = & 
  \left[\sum_{a=0}^{e-1}  x_{i,j+afn} ,
   \sum_{b=0}^{e-1}  x_{k,\ell+bfn} \right] \\
 & = & \sum_{a=0}^{e-1}\sum_{b=0}^{e-1} [x_{i,j+afn} , x_{k,\ell+bfn} ] \\
 & = & \left\{\begin{array}{ll} \sum_{a=0}^{e-1}\sum_{b=0}^{e-1} \left( \delta_{i+j+afn}^{\ell+bfn} x_{i+k,j+afn} - \delta_{k+\ell+bfn}^{j+afn}x_{i+k,\ell+bfn}\right) & \mbox{\ if\ } i+k\leq \frac{pdn}{p-1} \\ 0 & \mbox{\ if\ } i+k > \frac{pdn}{p-1} \end{array}\right. \\
 & = & \left\{\begin{array}{ll} \tilde{\delta}_{i+j}^{\ell} \sum_{a=0}^{e-1}x_{i+k,j+afn} - \tilde{\delta}_{k+\ell}^j \sum_{b=0}^{e-1} x_{i+k,\ell+bfn} 
 & \mbox{\ if\ } i+k\leq \frac{pdn}{p-1} \\ 0 & \mbox{\ if\ } i+k > \frac{pdn}{p-1} \end{array}\right. \\
 & = & \tilde{\delta}_{i+j}^{\ell}x_{i+k,j}^A - \tilde{\delta}_{k+\ell}^j x_{i+k,\ell}^A , 
\end{eqnarray*}
where $\delta^{b}_a=1$ if $a\equiv b$ modulo $fn$, and $\delta^b_a=0$ if $a\nequiv b$ modulo $fn$.

For the restriction on $PE_0k(\omega)[\strictAut({}_{\omega}\mathbb{G}_{1/n}^A)]$, we could proceed as above, computing 
\[ \xi\left(\sum_{\ell=0}^{e-1} 
                         x_{i,j+\ell fn}\right)\]
in $PE_0\mathbb{F}_p[\strictAut({}_1\mathbb{G}_{1/dn}^{\hat{\mathbb{Z}}_p})]$; but this immediately means contending with the non-linearity of the restriction map,
which complicates the computation. Instead it is easier to compute $\xi$ on
$PE_0k(\omega)[\strictAut({}_{\omega}\mathbb{G}_{1/n}^A)]$ in basically the same way that Ravenel computes $\xi$ on $PE_0\mathbb{F}_p[\strictAut({}_1\mathbb{G}_{1/dn}^{\hat{\mathbb{Z}}_p})]$ in the proof of
Proposition 6.3.3 of \cite{MR860042}; we sketch that method here. To compute
$\xi(x_{i,j}^A)$ we just need to find which elements $t_{a,b}\in E^0k(\omega)[\strictAut({}_{\omega}\mathbb{G}_{1/n}^A)]^*$ have the property
that the $(p-1)$st iterate $\Delta\circ \dots\circ \Delta (t_{a,b})$ of $\Delta$, applied to 
$t_{a,b}$, has a monomial term which is a scalar multiple of the $p$-fold tensor 
power $t_{i,j}\otimes \dots \otimes t_{i,j}$. When 
$i>\frac{dn}{p-1}$, then $i+dn>\frac{pdn}{p-1}$ and hence, by formula~\ref{coproduct on associated graded},
we have
\begin{eqnarray*} \Delta(t_{i+dn,j+1}) & = & t_{i+dn,j+1}\otimes 1 + 1\otimes t_{i+dn,j+1} 
                                              + \overline{b}_{i,j} \\
                                      & = & t_{i+dn,j+1}\otimes 1 + 1\otimes t_{i+dn,j+1} 
                - \sum_{0<\ell<p}\frac{1}{p}\binom{p}{\ell} t^{\ell}_{i,j}\otimes t_{i,j}^{p-\ell} ,
\end{eqnarray*}
and hence after $p-1$ iterations of $\Delta$ applied to $t_{i+dn,j+1}$, we get a copy of the 
monomial $t_{i,j}\otimes \dots \otimes t_{i,j}$. When $pi \leq \frac{pdn}{p-1}$ and 
$i=kfn$ for some positive integer $k$, then
formula~\ref{coproduct on associated graded} gives us that the $(p-1)$st iterate of
$\Delta$, applied to $\Delta(t_{pi,j})$, contains the monomial
\begin{eqnarray*} 
 & = & 
 t_{kfn,j}\otimes t_{kfn,j+kfn}\otimes t_{kfn,j+2kfn} \otimes \dots \otimes t_{kfn,j+(p-1)kfn} \\
 & = & t_{kfn,j}\otimes t_{kfn,j}\otimes \dots
 \otimes t_{kfn,j} 
\end{eqnarray*}
It is simple to show that no further monomials $t_{a,b}$ have the property that their $(p-1)$st
iterated coproducts contain the $p$th tensor power monomial $t_{i,j}\otimes \dots \otimes t_{i,j}$.
Formula~\ref{restriction formula} follows.

The relation $x_{i+n,fn}^A = x_{i+n,0}^A$ follows from the fact that
$x_{i+n,fn}^A$ is dual to $t_{i+n,fn}$ and $x_{i+n,0}^A$ is dual to $t_{i+n,0}$.
\end{proof}

\begin{definition} Let $K/\mathbb{Q}_p$ be a finite extension with degree $d$, ramification degree $e$, residue degree $f$, and let $A$ be its ring of integers, $\pi$ a uniformizer for $A$, and $k$ the residue field of $A$.
Let $n$ be a positive integer, and let $\omega$ be a $\left(\frac{q^{en}-1}{q^{n}-1}\right)$th root, 
in $\overline{k}$,
of $\frac{\pi^e}{p}$. We have the restricted graded Lie algebras 
$PE_0k(\omega)[\strictAut({}_{\omega}\mathbb{G}_{1/n}^A)]$
and $PE_0k(\omega)[\strictAut({}_{1}\mathbb{G}_{1/dn}^{\hat{\mathbb{Z}}_p})]$
over $k(\omega)$, and we write $L^A_{\omega}(n)$ and $L(dn)$, respectively, as shorthand notations for them.
If $\ell$ is a positive integer, we will also write
$L(dn,\ell)$ for the quotient Lie algebra of $L(dn)$ in which we 
quotient out all generators $x_{i,j}$ for which $i>\ell$ (this notation agrees with that of Theorem~\ref{ravenels spectral sequences}); and
we will write $L^A_{\omega}(n,\ell)$ for the quotient Lie algebra of $L^A_{\omega}(n)$ in which we 
quotient out all generators $x_{i,j}^A$ for which $i>\ell$. We have an obvious commutative diagram
of homomorphisms of Lie algebras:
\[ \xymatrix{ L^A_{\omega}(n,\ell) \ar[r] & L(dn,\ell)  \\
 L^A_{\omega}(n)\ar[r]\ar[u] & L(dn)\ar[u].}\]\end{definition}

Now here is a very useful corollary of Theorem~\ref{lie bracket computation}:
\begin{corollary}\label{coincidence of lie algebras}
The restricted Lie algebra $L^{\hat{\mathbb{Z}}_p[\sqrt[e]{p}]}_1(n,m)$ 
is isomorphic to the restricted Lie algebra $L^{\hat{\mathbb{Z}}_p}_1(n,m) = L(n,m)$ 
as long as $m\leq \frac{pn}{p-1}$. 
\end{corollary}

J.P. May actually constructed two different types of spectral sequence in his thesis~\cite{MR2614527}:
the spectral sequence of a filtered Hopf algebra, as in~\ref{may 2 ss}, is
the one most typically called a ``May spectral sequence.'' The other spectral sequence,
of Corollary~9 of~\cite{MR0193126} (as in~\ref{may 1 ss}), is the one which
computes restricted Lie algebra cohomology from unrestricted Lie algebra cohomology; we will
call that spectral sequence the
{\em Lie-May spectral sequence} to distinguish it from the May spectral sequence (unfortunately,
there is probably no perfect choice of terminology to be made here; 
e.g. chapter 6 of Ravenel's book \cite{MR860042}
refers to both spectral sequences as May spectral sequences).

\begin{theorem} \label{main thm on lie-may spect seqs for formal modules}
Let $K/\mathbb{Q}_p$ be a field extension of degree $d$ and ramification degree $e$ and residue degree $f$. Let $A$ be the ring of integers of $K$, let $\pi$ be a uniformizer for $A$, and let $k$ be the residue field of $A$. 
Let $n$ be a positive integer, and let $\omega \in \overline{k}$ be a 
$\frac{q^{en}-1}{q^n-1}$th root of $\frac{\pi^e}{p}$. 

We have the morphism of Lie-May spectral sequences
\[\label{morphism of lie-may spectral sequences}
\xymatrix{ H^*_{unr}(L(dn); k(\omega))\otimes_{k(\omega)} 
k(\omega)\left[\left\{ b_{i,j} : i\geq 1,  0\leq j\leq dn-1 \right\} \right]
 \ar@{=>}[r] \ar[d] &
  H^*_{res}(L(dn); k(\omega)) \ar[d] \\
H^*_{unr}(L^A_{\omega}(n); k(\omega))\otimes_{k(\omega)} 
k(\omega)\left[\left\{ b_{i,j} : f\mid i,  0\leq j\leq fn-1 \right\} \right]
 \ar@{=>}[r] &
  H^*_{res}(L^A_{\omega}(n);k(\omega)),} \]
with $b_{i,j}$ in bidegree $(2, 0)$ (these two gradings are, respectively, cohomological degree
and Lie-May degree) and with auxiliary bidegree $(p \| t_{i,j}\| ,2p^{j+1}(p^i-1))$ (these two gradings are, respectively, the grading coming from the Ravenel filtration, and the grading coming from the topological grading on $BP_*BP$) in each spectral sequence,
where $\| t_{i,j}\|$ is the Ravenel degree of $t_{i,j}$. 
The elements in $H^t_{unr}(L^A_{\omega}(n); k(\omega))$ are in bidegree $(0,t)$. The
differential is, as is typical for the spectral sequence of a filtered cochain complex,
$d_r^{s,t}: E_r^{s,t}\rightarrow E_r^{s+r,t-r+1}$.

If $p = \pi^e$, then 
we have a tensor splitting of each of these Lie-May spectral sequences, such that the splittings
are respected by the morphism \ref{morphism of lie-may spectral sequences}
of spectral sequences:
the Lie-May spectral sequence \[ H^*_{unr}(L(dn); \mathbb{F}_q)\otimes_{\mathbb{F}_q} 
\mathbb{F}_q\left[\left\{ b_{i,j} : i \geq 1, 0\leq j\leq dn-1 \right\} \right]
\Rightarrow 
  H^*_{res}(L(dn); \mathbb{F}_q)\]
splits into a tensor product of a spectral sequence
\[ H^*_{unr}(L(dn,\frac{pdn}{p-1}); \mathbb{F}_q)\otimes_{\mathbb{F}_q} 
\mathbb{F}_q\left[\left\{ b_{i,j} : 1\leq i \leq \frac{dn}{p-1}, 0\leq j\leq dn-1 \right\} \right]
\Rightarrow 
  H^*_{res}(L(dn, \frac{pdn}{p-1}; \mathbb{F}_q)\]
with a spectral sequence with trivial $E_\infty$-term; and likewise, the Lie-May spectral sequence
\[ H^*_{unr}(L^A_{1}(n); \mathbb{F}_q)\otimes_{\mathbb{F}_q} 
\mathbb{F}_q\left[\left\{ b_{i,j} : f\mid i , 0\leq j\leq fn-1 \right\} \right]
\Rightarrow 
  H^*_{unr}(L^A_{1}(n); \mathbb{F}_q)\]
splits into a tensor product of a spectral sequence
\[ H^*_{unr}(L^A(n,\frac{pdn}{p-1}); \mathbb{F}_q)\otimes_{\mathbb{F}_q} 
\mathbb{F}_q\left[\left\{ b_{i,j} : f\mid i, 1\leq i \leq \frac{dn}{p-1}, 0\leq j\leq fn-1 \right\} \right]
\Rightarrow 
  H^*_{res}(L^A(n,\frac{pdn}{p-1}); \mathbb{F}_q)\]
with a spectral sequence with trivial $E_\infty$-term.

We have a morphism of spectral sequences:
\[\label{map of reduced lie-may spectral sequences}
\xymatrix{ H^*_{unr}(L(dn,\frac{pdn}{p-1}); \mathbb{F}_q)\otimes_{\mathbb{F}_q} 
\mathbb{F}_q\left[\left\{ b_{i,j} : 1\leq i \leq \frac{dn}{p-1}, 0\leq j\leq dn-1 \right\} \right]
 \ar@{=>}[r] \ar[d] &
  H^*_{res}(L(dn,\frac{pdn}{p-1}); \mathbb{F}_q) \ar[d] \\
H^*_{unr}(L^A_{1}(n,\frac{pdn}{p-1}); \mathbb{F}_q)\otimes_{\mathbb{F}_q} 
\mathbb{F}_q\left[\left\{ b_{i,j} : f\mid i, i\leq \frac{dn}{p-1}, 0\leq j\leq fn-1 \right\} \right]
 \ar@{=>}[r] &
  H^*_{res}(L^A_{1}(n,\frac{pdn}{p-1}); \mathbb{F}_q).} \]
\end{theorem}
\begin{proof}
That the morphism~\ref{morphism of lie-may spectral sequences} of spectral sequences
exists follows from May's construction of the Lie-May spectral sequence in \cite{MR0193126}.
The splittings occur because of formula~\ref{bracket formula}, which tells us that
the unrestricted Lie algebra underlying $L(dn)$ splits into a product of $L(dn,\frac{pdn}{p-1})$ with an abelian Lie algebra
generated by $\left\{ x_{i,j}: i>\frac{pdn}{p-1} \right\}$, and the unrestricted Lie algebra underlying 
$L^A_1(n)$ splits into a product of $L^A(n,\frac{pdn}{p-1})$ with an abelian Lie algebra
generated by $\left\{ x^A_{i,j}: i>\frac{pdn}{p-1}\right\}$;
and formula~\ref{formula for iota} tells us that the morphism $L^A_1(n)\hookrightarrow
L(dn)$ respects these product splittings.
By formula~\ref{restriction formula}, 
the restriction on $PE_0\mathbb{F}_p[\strictAut({}_1\mathbb{G}_{1/dn}^{\hat{\mathbb{Z}}_p})]$ sends $x_{i,j}$ to $x_{i+n,j+1}$ when $i> \frac{dn}{p-1}$;
so the filtered chain complex (see Theorem~5 in~\cite{MR0185595} or Corollary~9 of \cite{MR0193126}) 
whose associated spectral
sequence is the Lie-May spectral sequence has the property that it splits into a tensor 
product of a cohomologically trivial filtered chain complex and one whose associated graded
chain complex has cohomology $H^*_{unr}(L(dn,\frac{pdn}{p-1}); \mathbb{F}_q)\otimes_{\mathbb{F}_q} \mathbb{F}_q\left[\left\{ b_{i,j} : 1\leq i \leq \frac{dn}{p-1}, 0\leq j\leq dn-1 \right\}\right]$.
An analogous statement holds for the bracket and restriction on $PE_0\mathbb{F}_q[\strictAut({}_{1}\mathbb{G}_{1/n}^{A})]$ and the
filtered chain complex giving its Lie-May spectral sequence.
\end{proof}

\begin{definition}
Let $K/\mathbb{Q}_p$ be a field extension of degree $d$ and ramification degree $e$ and residue degree $f$. Let $A$ be the ring of integers of $K$, let $\pi$ be a uniformizer for $A$, and let $k$ be the residue field of $A$. 
Let $n$ be a positive integer, and let $\omega \in \overline{k}$ be a 
$\frac{q^{en}-1}{q^n-1}$th root of $\frac{\pi^e}{p}$. 
We write $\mathcal{K}^A_{\omega}(n)$ for the Chevalley-Eilenberg DGA of the Lie algebra $L^A_{\omega}(n)$.
If $m$ is a positive integer, we write 
$\mathcal{K}^A_{\omega}(n,m)$ for the Chevalley-Eilenberg DGA of the Lie algebra $L^A_{\omega}(n,m)$.
(Note that the Chevalley-Eilenberg DGA depends only on the underlying unrestricted Lie algebra.)

The cyclic group $C_{dn}$ acts on $\mathcal{K}^A_{\omega}(n)$ by sending $h_{i,j}$ to $h_{i,j+1}$, and when $\omega = 1$, this action reduces to an action of $C_n$ on $\mathcal{K}^A_{\omega}(n)$.

The DGAs $\mathcal{K}^A_{\omega}(n)$ and $\mathcal{K}^A_{\omega}(n,m)$ are equipped with several gradings which we will need to keep track of: the cohomological grading; the topological grading (sometimes also called the ``internal grading'') inherited from $BP_*BP$, which is only defined modulo $2(p^{fn}-1)$; and the Ravenel grading, inherited from the Ravenel filtration.
Note that the $C_{dn}$-action preserves the cohomological gradings and the Ravenel grading, but not the internal grading; this behavior will be typical in all of the multigraded DGAs we consider, and we adopt the convention that, whenever we speak of a ``multigraded equivariant DGA,'' we assume that the group action preserves all of the gradings except possibly the internal grading.
\end{definition}

The presentation in Theorem~\ref{ravenels chevalley-eilenberg computation} generalizes as follows:
\begin{observation}\label{def of K dga}
It is easy and routine to extract a presentation for the Chevalley-Eilenberg DGA from Proposition~\ref{assoc graded with complex mult}, without using the formulas in Theorem~\ref{lie bracket computation}:
$\mathcal{K}^A_{\omega}(n)$ is the exterior algebra (over $k(\omega)$) with generators given by the set of symbols $h_{i,j}$ with $i$ divisible by the residue degree $f$ and satisfying $1\leq i$, and $j\in \mathbb{Z}/fn\mathbb{Z}$; the differential is given by
\[ d(h_{i,j}) = \sum_{k=1}^{i-1} h_{k,j} h_{i-k,j+k},\]
with the convention that $h_{i,k+fn} = \omega^{p^k(q^i-1)}h_{i,k}$. 
Similarly, $\mathcal{K}^A_{\omega}(n,m)$ is the sub-DGA of $\mathcal{K}^A_{\omega}(n)$ generated by all $h_{i,j}$ with $i\leq m$.

When $\omega = 1$ and $A = \hat{\mathbb{Z}}_p[\sqrt[e]{p}]$ and $m\leq \frac{pn}{p-1}$,
we write $\mathcal{K}(n,m)$ as shorthand for $\mathcal{K}^{\hat{\mathbb{Z}}_p[\sqrt[e]{p}]}_1(n,m)$ and $\mathcal{K}^{\hat{\mathbb{Z}}_p}_1(n,m)$; this notation is unambiguous because of Corollary~\ref{coincidence of lie algebras}.
\end{observation}

\section{Cohomology computations.}\label{computations section}

\subsection{The cohomology of $\mathcal{K}(2,2)$ and the height $2$ Morava stabilizer group scheme.}

The material in this subsection is easy and well-known, appearing already in section~6.3 of~\cite{MR860042}. 
\begin{prop}\label{coh of K 2 2}
Suppose $p>2$.
Then we have an isomorphism of trigraded $C_2$-equivariant $\mathbb{F}_p$-algebras
\[ H^{*,*,*}(L(2,2)) \cong \mathbb{F}_p\{ 1,h_{10},h_{11},h_{10}\eta_2,h_{11}\eta_2,h_{10}h_{11}\eta_2 \} \otimes_{\mathbb{F}_p} \Lambda(\zeta_2), \]
with tridegrees and the $C_2$-action as follows (remember that the internal degree is always reduced modulo $2(p^2-1)$):
\begin{equation}\label{degree chart 1}
\begin{array}{llllll}
\mbox{Coh.\ class}          & \mbox{Coh.\ degree} & \mbox{Int.\ degree} & \mbox{Rav.\ degree} & \mbox{Image\ under\ } \sigma \\
1                           & 0                   & 0                   & 0                   & 1  \\
h_{10}                      & 1                   & 2(p-1)              & 1                   & h_{11} \\
h_{11}                      & 1                   & 2p(p-1)             & 1                   & h_{10} \\
\zeta_2                    & 1                   & 0                   & 2                   & \zeta_2 \\
h_{10}\eta_2                & 2                   & 2(p-1)              & 3                   & -h_{11}\eta_2 \\
h_{11}\eta_2                & 2                   & 2p(p-1)             & 3                   & -h_{10}\eta_2 \\
h_{10}\zeta_2                & 2                   & 2(p-1)              & 3                   & h_{11}\zeta_2 \\
h_{11}\zeta_2                & 2                   & 2p(p-1)             & 3                   & h_{10}\zeta_2 \\
h_{10}h_{11}\eta_2          & 3                   & 0                   & 4                   & h_{10}h_{11}\eta_2 \\
h_{10}\eta_2\zeta_2         & 3                   & 2(p-1)              & 5                   & -h_{11}\eta_2\zeta_2 \\
h_{11}\eta_2\zeta_2         & 3                   & 2p(p-1)             & 5                   & -h_{10}\eta_2\zeta_2 \\
h_{10}h_{11}\eta_2\zeta_2    & 4                   & 0                   & 6                   & h_{10}h_{11}\eta_2\zeta_2 .
    \end{array} \end{equation}
where the cup products in $\mathbb{F}_p\{ 1,h_{10},h_{11},h_{10}\eta_2,h_{11}\eta_2,h_{10}h_{11}\eta_2 \}$ are all zero aside from the Poincar\'{e} duality cup products, i.e.,
each class has the obvious dual class such that the cup product of the two is
$h_{10}h_{11}\eta_2$, and the remaining cup products are all zero.
\end{prop}
\begin{proof}
We have the extension of Lie algebras
\[ 1 \rightarrow \mathbb{F}_p\{ x_{20},x_{21}\} \rightarrow L(2,2) \rightarrow L(2,1) \rightarrow 1,\]
and to compute the resulting spectral sequence in cohomology, 
we take the Chevalley-Eilenberg DGAS and then compute the Cartan-Eilenberg spectral sequence for the extension of $C_2$-equivariant trigraded DGAs
\[ 1 \rightarrow \mathcal{K}(2,1) \rightarrow
  \mathcal{K}(2,2)\rightarrow
  \Lambda(h_{20},h_{21}) \rightarrow 1.\]
Since the differential on $\mathcal{K}(2,1)$ is zero (see Observation~\ref{def of K dga}), $H^{*,*,*}(\mathcal{K}(2,1)) \cong \mathcal{K}(2,1) \cong \Lambda(h_{10},h_{11})$.
A change of $\mathbb{F}_p$-linear basis is convenient here: we will write $\zeta_2$ for the element $h_{20} + h_{21}\in \Lambda(h_{20},h_{21})$. (This notation for this particular element is standard. As far as I know, it began with M. Hopkins' work on the ``chromatic splitting conjecture,'' in which $\zeta_2$ plays a special role.)
We will write $\eta_2$ for the element $h_{20} - h_{21}$.

We have the differentials
\begin{eqnarray*} d\zeta_2 & = & 0, \\
 d\eta_2 & = & -2h_{10}h_{11}\end{eqnarray*}
\end{proof}
The table~\ref{degree chart 1} has one row for each element in an $\mathbb{F}_p$-linear basis for the cohomology ring $H^{*,*,*}(\mathcal{K}(2,2))$, but from now on in this document, for the sake of brevity, when writing out similar tables for grading degrees of elements in the cohomology of a multigraded equivariant DGA, 
I will just give one row for each element in a set of generators for the cohomology ring of the DGA.

\begin{prop}\label{cohomology of ht 2 morava stab grp}
Suppose $p>3$. Then the cohomology $H^*(\strictAut({}_1\mathbb{G}_{1/2}); \mathbb{F}_p)$ of the height $2$ strict Morava stabilizer group scheme is isomorphic, as a graded $\mathbb{F}_p$-vector space, to
\[ H^{*,*,*}(\mathcal{K}(2,2)) \cong \mathbb{F}_p\{ 1,h_{10},h_{11},h_{10}\eta_2,h_{11}\eta_2,h_{10}h_{11}\eta_2 \} \otimes_{\mathbb{F}_p} \Lambda(\zeta_2) \]
from Proposition~\ref{coh of K 2 2}.
The cohomological grading on $H^*(\strictAut(\mathbb{G}_{1/2}); \mathbb{F}_p)$ corresponds to the cohomological grading on $H^{*,*,*}(\mathcal{K}(2,2))$, so that $h_{10},h_{11},\zeta_2 \in H^1(\strictAut(\mathbb{G}_{1/2}); \mathbb{F}_p)$,
$h_{10}\eta_2,h_{11}\eta_2\in H^2(\strictAut(\mathbb{G}_{1/2}); \mathbb{F}_p)$, and so on.

The multiplication on $H^*(\strictAut(\mathbb{G}_{1/2}); \mathbb{F}_p)$ furthermore agrees with the multiplication on $H^{*,*,*}(\mathcal{K}(2,2))$, modulo the question of exotic multiplicative extensions, i.e., jumps in Ravenel filtration in the products of elements in $H^*(\strictAut(\mathbb{G}_{1/2}); \mathbb{F}_p)$. 
\end{prop}
\begin{proof}
Spectral sequence~\ref{may 1 ss} collapses immediately, since $p>3$ implies that
$1> \floor{\frac{2}{p-1}}$.
Hence $\Cotor_{E^0\mathbb{F}_p[\strictAut(\mathbb{G}_{1/2})]^*}^{*,*,*}(\mathbb{F}_p,\mathbb{F}_p) \cong H^{*,*,*}(\mathcal{K}(4,2))$.

We now run spectral sequence~\ref{may 2 ss}. This is, like all May spectral sequences, the spectral sequence of the filtration (in this case, Ravenel's filtration) on the cobar complex $C^{\bullet}(A)$ of a coalgebra $A$ induced by a filtration on the coalgebra itself. To compute differentials, we take an element $x\in H^*(C^{\bullet}(E^0A))$, lift it to a cochain $\overline{x}\in H^*(C^{\bullet}(A))$ whose image in the cohomology of the associated graded $H^*(E^0(C^{\bullet}(A))) \cong H^*(C^{\bullet}(E^0A))$ is $x$, and then evaluate the differential $d(\overline{x})$ in the cobar complex $C^{\bullet}(A)$. If $d(\overline{x}) = 0$, then $\overline{x}$ is a cocycle in the cobar complex $C^{\bullet}(A)$ and not merely in its associated graded $E^0C^{\bullet}(A)$, hence $\overline{x}$ represents a cohomology class in $H^*(C^{\bullet}(A))$; if $d(\overline{x}) \neq 0$, then we add correcting coboundaries of lower or higher (depending on whether the filtration is increasing or decreasing) filtration until we arrive at a cocycle which we recognize as a cohomology class in the spectral sequence's $E_1$-page.

It will be convenient to use the presentation 
\[ \mathbb{F}_p[t_{i,j}: i\geq 1, 0\leq j\leq 1]/\left( t_{i,j}^p \mbox{\ \ for\ all\ } i,j\right) \]
for 
$E^0\left(\mathbb{F}_p[\strictAut({}_1\mathbb{G}_{1/2})]^*\right)  \cong E^0\left( \mathbb{F}_p[t_1, t_2, \dots ]/\left( t_i^{p^2} - t_i\mbox{\ \ for\ all\ } i\right)\right)$, where $t_{i,j}$ is the image in the associated graded of $t_i^{p^j}$.
The coproduct on $\mathbb{F}_p[t_{i,j}: i\geq 1, 0\leq j\leq 1]/\left( t_{i,j}^p \mbox{\ \ for\ all\ } i,j\right)$, inherited from that of $\mathbb{F}_p[\strictAut(\mathbb{G}_{1/2})]^*$, is given by
\[ \Delta(t_{i,j}) = \sum_{k=0}^i t_{k,j} \otimes t_{i-k, k+j}\]
for all $i < \floor{\frac{2p}{p-1}}$; see Theorem~6.3.2 of~\cite{MR860042} for this formula.
\begin{description}
\item[$h_{10},h_{11}$] The class $h_{10}$ is represented by $t_{1,0}$ in the cobar complex $C^{\bullet}\left(E^0\left(\mathbb{F}_p[\strictAut({}_1\mathbb{G}_{1/2})]^*\right)\right)$, which lifts to $t_1$ in the cobar complex 
$C^{\bullet}\left(\mathbb{F}_p[\strictAut({}_1\mathbb{G}_{1/2})]^*\right)$.
Since $t_1$ is a coalgebra primitive, i.e., a cobar complex $1$-cocycle, 
all May differentials are zero on $h_{1,0}$.
The $C_2$-equivariance of the spectral sequence then tells us that all May 
differentials also vanish on $h_{1,0}$.
\item[$\zeta_2$] There is no nonzero class in cohomological degree $2$ and internal degree $0$ for $\zeta_2$ to hit by a May differential of any length.
\item[$h_{10}\eta_2, h_{11}\eta_2$] The cohomology class $h_{10}\eta_2$ in the Chevalley-Eilenberg complex of the Lie algebra of primitives in $E^0\left(\mathbb{F}_p[\strictAut({}_1\mathbb{G}_{1/2})]^*\right)$ (of which $\mathcal{K}(4,2)$ is a subcomplex) is represented by the $2$-cocycle $t_{1,0}\otimes t_{2,0} - t_{1,0}\otimes t_{2,1} - t_{1,0}\otimes t_{1,0}t_{1,1}$ in the cobar complex of $E^0\left(\mathbb{F}_p[\strictAut({}_1\mathbb{G}_{1/2})]^*\right)$. This $2$-cocycle lifts to the $2$-cocycle
$t_{1}\otimes t_{2} - t_{1}\otimes t_{2}^p - t_{1}\otimes t_{1}^{p+1}$ in the cobar complex of $\mathbb{F}_p[\strictAut({}_1\mathbb{G}_{1/2})]^*$.
Hence all May differentials vanish on $h_{10}\eta_2$, and by $C_2$-equivariance, also $h_{11}\eta_2$.
\end{description}
So the May differentials of all lengths vanish on the generators of the ring
$\Cotor_{E^0\mathbb{F}_p[\strictAut({}_1\mathbb{G}_{1/2})]^*}^{*,*,*}(\mathbb{F}_p,\mathbb{F}_p)$. So 
$H^*(\strictAut({}_1\mathbb{G}_{1/2}); \mathbb{F}_p) \cong \Cotor_{E^0\mathbb{F}_p[\strictAut({}_1\mathbb{G}_{1/2})]^*}^{*,*,*}(\mathbb{F}_p,\mathbb{F}_p) \cong H^{*,*,*}(\mathcal{K}(2,2))$ as a graded $\mathbb{F}_p$-vector space. 
\end{proof}

\subsection{The cohomology of $\mathcal{K}(2,3),\mathcal{K}(2,4)$, and the automorphism group scheme of a $\hat{\mathbb{Z}}_p[\sqrt{p}]$-height $2$ formal $\hat{\mathbb{Z}}_p[\sqrt{p}]$-module.}

\begin{prop}\label{coh of E 4 3 0}
Suppose $p>3$.
Then we have an isomorphism of trigraded $C_2$-equivariant $\mathbb{F}_p$-algebras
\[ H^{*,*,*}(\mathcal{K}(2,3)) \cong \mathcal{A}_{2,3}\otimes_{\mathbb{F}_p} \Lambda(\zeta_2), \]
where 
\begin{dmath*} \mathcal{A}_{2,3} \cong 
\mathbb{F}_p\{ 1,h_{10},h_{11},h_{10}h_{30}, h_{11}h_{31}, e_{40}, \eta_2e_{40}, h_{10}\eta_2h_{30}, h_{11}\eta_2h_{31}, h_{10}\eta_2h_{30}h_{31}, h_{11}\eta_2h_{30}h_{31}, h_{10}h_{11}\eta_2h_{30}h_{31}\} ,\end{dmath*}
with tridegrees and the $C_2$-action as follows: 
\begin{equation}\label{degree chart 3}
\begin{array}{llllll}
\mbox{Coh.\ class}          & \mbox{Coh.\ degree} & \mbox{Int.\ degree} & \mbox{Rav.\ degree} & \mbox{Image\ under\ } \sigma \\
\hline \\
1                           & 0                   & 0                   & 0                   & 1  \\
h_{10}                      & 1                   & 2(p-1)              & 1                   & h_{11} \\
h_{11}                      & 1                   & 2p(p-1)             & 1                   & h_{10} \\
h_{10}h_{30}                 & 2                   & 4(p-1)              & 1+p                   & h_{11}h_{31} \\
h_{11}h_{31}                 & 2                   & 4p(p-1)             & 1+p                   & h_{10}h_{30} \\
e_{40}                      & 2                   & 0                   & 1+p                   & -e_{40}\\
\eta_2e_{40}                & 3                   & 0                   & 3+p                   & \eta_2e_{40}\\
h_{10}\eta_2h_{30}           & 3                   & 4(p-1)              & 3+p                   & -h_{11}\eta_2h_{31} \\
h_{11}\eta_2h_{31}           & 3                   & 4p(p-1)             & 3+p                   & -h_{10}\eta_2h_{30} \\
h_{10}\eta_2h_{30}h_{31}     & 4                   & 2(p-1)              & 3+2p                  & h_{11}\eta_2h_{30}h_{31} \\
h_{11}\eta_2h_{30}h_{31}     & 4                   & 2p(p-1)             & 3+2p                   & h_{10}\eta_2h_{30}h_{31} \\
h_{10}h_{11}\eta_2h_{30}h_{31} & 5                   & 0                 & 4+2p                   & -h_{10}h_{11}\eta_2h_{30}h_{31} \\
\hline \\
\zeta_2                    & 1                   & 0                   & 2                   & \zeta_2 ,
    \end{array} \end{equation}
where the cup products in $\mathcal{A}_{2,3}$ are all zero aside from the Poincar\'{e} duality cup products, i.e.,
each class has the dual class such that the cup product of the two is
$h_{10}h_{11}\eta_2h_{30}h_{31}$, and the remaining cup products are all zero. The classes in table~\ref{degree chart 3} are listed in order so that the class which is $n$ lines below $1$ is, up to multiplication by a unit in $\mathbb{F}_p$, the Poincar\'{e} dual of the class which is $n$ lines above $h_{10}h_{11}\eta_2h_{30}h_{31}$.
\end{prop}
\begin{proof}
We have the extension of Lie algebras
\[ 1 \rightarrow \mathbb{F}_p\{ x_{30},x_{31}\} \rightarrow L(2,3) \rightarrow L(2,2) \rightarrow 1 \]
and
we take their Chevalley-Eilenberg DGAs, then compute the Cartan-Eilenberg spectral sequence for the extension of $C_2$-equivariant trigraded DGAs:
\[ 1 \rightarrow \mathcal{K}(2,2) \rightarrow
  \mathcal{K}(2,3)\rightarrow
  \Lambda(h_{30},h_{31}) \rightarrow 1.\]
We have the differentials
\begin{eqnarray*} dh_{30} & = & -h_{10}\eta_2, \\
 dh_{31} & = & h_{11}\eta_2,\\
 d(h_{30}h_{31}) & = & -h_{10}\eta_2h_{31} - h_{11}\eta_2h_{30}.
\end{eqnarray*}
and their products with classes in $H^{*,*,*}(\mathcal{K}(2,2))$.
The nonzero products are
\begin{eqnarray*} 
 d(h_{11}h_{30}) & = & -h_{10}h_{11}\eta_2 \\
 d(h_{10}h_{31}) & = & -h_{10}h_{11}\eta_2 \\
 d(h_{10}h_{30}h_{31}) & = & h_{10}h_{11}\eta_2h_{30} \\
 d(h_{11}h_{30}h_{31}) & = & -h_{10}h_{11}\eta_2h_{31}. 
\end{eqnarray*}
We write $e_{40}$ for the cocycle $h_{10}h_{31} - h_{11}h_{30}$.
Extracting the output of the spectral sequence from knowledge of the differentials is routine. 
\end{proof}

\begin{prop}\label{coh of E 4 4 0}
Suppose $p>3$.
Then we have an isomorphism of trigraded $C_2$-equivariant $\mathbb{F}_p$-algebras
\[ H^{*,*,*}(\mathcal{K}(2,4)) \cong \mathcal{A}_{2,4} \otimes_{\mathbb{F}_p} \Lambda(\zeta_2,\zeta_4), \]
where 
\begin{dmath*}
\mathcal{A}_{2,4} \cong \mathbb{F}_p\{ 1,h_{10},h_{11},h_{10}h_{30}, h_{11}h_{31}, h_{10}\eta_4-\eta_2h_{30}, h_{11}\eta_4-\eta_2h_{31}, \eta_2e_{40}, h_{10}\eta_2h_{30}, h_{11}\eta_2h_{31}, h_{10}h_{30}\eta_4, h_{11}h_{31}\eta_4, \eta_4e_{40}+4\eta_2h_{30}h_{31}, h_{10}\eta_2h_{30}h_{31}, h_{11}\eta_2h_{30}h_{31}, h_{10}\eta_2h_{30}\eta_4, h_{11}\eta_2h_{31}\eta_4, h_{10}\eta_2h_{30}h_{31}\eta_4, h_{11}\eta_2h_{30}h_{31}\eta_4, h_{10}h_{11}\eta_2h_{30}h_{31}\eta_4\} ,\end{dmath*}
with tridegrees and the $C_2$-action as follows: 
\begin{equation}\label{degree chart 4}
\begin{array}{llllll}
\mbox{Coh.\ class}          & \mbox{Coh.\ degree} & \mbox{Int.\ degree} & \mbox{Rav.\ degree} & \mbox{Image\ under\ } \sigma \\
\hline \\
1                                 & 0                   & 0                   & 0                   & 1  \\
h_{10}                            & 1                   & 2(p-1)              & 1                   & h_{11} \\
h_{11}                            & 1                   & 2p(p-1)             & 1                   & h_{10} \\
h_{10}h_{30}                       & 2                   & 4(p-1)              & 1+p                   & h_{11}h_{31} \\
h_{11}h_{31}                       & 2                   & 4p(p-1)             & 1+p                   & h_{10}h_{30} \\
h_{10}\eta_4-\eta_2h_{30}          & 2                   & 2(p-1)              & 1+2p                 & -h_{11}\eta_4+\eta_2h_{30} \\
h_{11}\eta_4-\eta_2h_{31}          & 2                   & 2p(p-1)             & 1+2p                 & -h_{10}\eta_4+\eta_2h_{31} \\
\eta_2e_{40}                      & 3                   & 0                   & 3+p                   & \eta_2e_{40}\\
h_{10}\eta_2h_{30}                 & 3                   & 4(p-1)              & 3+p                   & -h_{11}\eta_2h_{31} \\
h_{11}\eta_2h_{31}                 & 3                   & 4p(p-1)             & 3+p                   & -h_{10}\eta_2h_{30} \\
h_{10}h_{30}\eta_4                 & 3                   & 4(p-1)              & 1+3p                  & -h_{11}h_{31}\eta_4 \\
h_{11}h_{31}\eta_4                 & 3                   & 4p(p-1)             & 1+3p                  & -h_{10}h_{30}\eta_4 \\
\eta_4e_{40}+4\eta_2h_{30}h_{31}   & 3                   & 0                   & 1+3p                   & \eta_4e_{40}+4\eta_2h_{30}h_{31}\\
h_{10}\eta_2h_{30}h_{31}           & 4                   & 2(p-1)              & 3+2p                  & h_{11}\eta_2h_{30}h_{31} \\
h_{11}\eta_2h_{30}h_{31}           & 4                   & 2p(p-1)             & 3+2p                   & h_{10}\eta_2h_{30}h_{31} \\
h_{10}\eta_2h_{30}\eta_4           & 4                   & 4(p-1)              & 3+3p                  & h_{11}\eta_2h_{31}\eta_4 \\
h_{11}\eta_2h_{31}\eta_4           & 4                   & 4p(p-1)             & 3+3p                  & h_{10}\eta_2h_{30}\eta_4 \\
h_{10}\eta_2h_{30}h_{31}\eta_4      & 5                   & 2(p-1)              & 3+4p                 & -h_{11}\eta_2h_{30}h_{31}\eta_4 \\
h_{11}\eta_2h_{30}h_{31}\eta_4      & 5                   & 2p(p-1)             & 3+4p                 & -h_{10}\eta_2h_{30}h_{31}\eta_4 \\
h_{10}h_{11}\eta_2h_{30}h_{31}\eta_4 & 6                   & 0                  & 4+4p                 & h_{10}h_{11}\eta_2h_{30}h_{31}\eta_4 \\
\hline \\
\zeta_2                          & 1                   & 0                   & 2                   & \zeta_2 \\
\zeta_4                          & 1                   & 0                   & 2p                  & \zeta_4 .
    \end{array} \end{equation}
The classes in table~\ref{degree chart 4} are listed in order so that the class which is $n$ lines below $1$ is, up to multiplication by a unit in $\mathbb{F}_p$, the Poincar\'{e} dual of the class which is $n$ lines above $h_{10}h_{11}\eta_2h_{30}h_{31}\eta_4$.
\end{prop}
\begin{proof}
We
compute the Cartan-Eilenberg spectral sequence for the extension of $C_2$-equivariant trigraded DGAs
\[ 1 \rightarrow \mathcal{K}(2,3) \rightarrow
  \mathcal{K}(2,4)\rightarrow
  \Lambda(h_{40},h_{41}) \rightarrow 1\]
arising from the extension of Lie algebras
\[ 1 \rightarrow \mathbb{F}_p\{ x_{40},x_{41}\} \rightarrow L(2,4) \rightarrow L(2,3) \rightarrow 1 .\]
A change of $\mathbb{F}_p$-linear basis is convenient here: we will write $\zeta_4$ for the element $h_{40} + h_{41}\in \Lambda(h_{40},h_{41})$, and we will write $\eta_4$ for $h_{40}-h_{41}$.
We have the differentials
\begin{eqnarray*} d\zeta_{4} & = & 0, \\
 d\eta_4 & = & h_{10}h_{31} + h_{30}h_{11} = e_{40},
\end{eqnarray*}
and a nonzero product with a class in $H^{*,*,*}(\mathcal{K}(2,3))$,
\begin{eqnarray*} 
 d(\eta_2e_{40}\eta_4) & = & h_{10}h_{11}\eta_2h_{30}h_{31}.
\end{eqnarray*}
Extracting the output of the spectral sequence from knowledge of the differentials is routine.
The three classes $h_{10}\eta_4, h_{11}\eta_4, \eta_4e_{40}$ in the $E_{\infty}$-term are not cocycles in $H^{*,*,*}(\mathcal{K}(2,4))$; adding
terms of lower Cartan-Eilenberg filtration to get cocycles yields the cohomology classes
$h_{10}\eta_4-\eta_2h_{30}, h_{11}\eta_4-\eta_2h_{31}, \eta_4e_{40}+4\eta_2h_{30}h_{31}$.
Note that this implies that there are nonzero multiplications in $\mathcal{A}_{2,4}$ other than those between each class and its Poincar\'{e} dual; for example, $h_{10}(h_{10}\eta_4 - \eta_{2}h_{30}) = -h_{10}\eta_2h_{30}$. 
\end{proof}

\begin{theorem}\label{coh of ht 2 fm}
Suppose $p>5$. Then the cohomology $H^*(\strictAut({}_1\mathbb{G}_{1/2}^{\hat{\mathbb{Z}}_p\left[\sqrt{p}\right]}); \mathbb{F}_p)$ of the
strict automorphism of the $\hat{\mathbb{Z}}_p\left[\sqrt{p}\right]$-height $2$ formal $\hat{\mathbb{Z}}_p\left[\sqrt{p}\right]$-module ${}_1\mathbb{G}^{\hat{\mathbb{Z}}_p\left[\sqrt{p}\right]}_{1/2}$ is isomorphic, as a graded $\mathbb{F}_p$-vector space, to
\[ H^{*,*,*}(\mathcal{K}(2,4)) \cong\mathcal{A}_{2,4} \otimes_{\mathbb{F}_p} \Lambda(\zeta_2,\zeta_4), \]
from Proposition~\ref{coh of E 4 4 0}.
The cohomological grading on $H^*(\strictAut({}_1\mathbb{G}_{1/2}^{\hat{\mathbb{Z}}_p\left[\sqrt{p}\right]}); \mathbb{F}_p)$ corresponds to the cohomological grading on $H^{*,*,*}(\mathcal{K}(2,4))$.

The multiplication on $H^*(\strictAut({}_1\mathbb{G}_{1/2}^{\hat{\mathbb{Z}}_p\left[\sqrt{p}\right]}); \mathbb{F}_p)$ furthermore agrees with the multiplication on $H^{*,*,*}(\mathcal{K}(2,4))$, modulo the question of exotic multiplicative extensions, i.e., jumps in Ravenel filtration in the products of elements in $H^*(\strictAut({}_1\mathbb{G}_{1/2}^{\hat{\mathbb{Z}}_p\left[\sqrt{p}\right]}); \mathbb{F}_p)$. 

In particular, the Poincar\'{e} series expressing the $\mathbb{F}_p$-vector space dimensions of the grading degrees in $H^*(\strictAut({}_1\mathbb{G}_{1/2}^{\hat{\mathbb{Z}}_p\left[\sqrt{p}\right]}); \mathbb{F}_p)$ is \[ (1+s)^2\left( 1 + 2s + 4s^2 + 6s^3 + 4s^4 + 2s^5 + s^6\right).\]
\end{theorem}
\begin{proof}
The Lie-May spectral sequence of Theorem~\ref{main thm on lie-may spect seqs for formal modules} collapses immediately, since $p>5$ implies that
$1> \floor{\frac{4}{p-1}}$.
Hence $\Cotor_{E^0\mathbb{F}_p[\strictAut({}_1\mathbb{G}^{\hat{\mathbb{Z}}_p\left[\sqrt{p}\right]}_{1/2})]^*}^{*,*,*}(\mathbb{F}_p,\mathbb{F}_p) \cong H^{*,*,*}(\mathcal{K}(2,4))$.

We now run the May spectral sequence 
\[ E_1^{*,*,*} \cong \Cotor_{E^0\mathbb{F}_p[\strictAut({}_1\mathbb{G}^{\hat{\mathbb{Z}}_p\left[\sqrt{p}\right]}_{1/2})]^*}^{*,*,*}(\mathbb{F}_p,\mathbb{F}_p) \Rightarrow H^*(\strictAut({}_1\mathbb{G}^{\hat{\mathbb{Z}}_p\left[\sqrt{p}\right]}_{1/2}); \mathbb{F}_p).\] 
See the proof of Proposition~\ref{cohomology of ht 2 morava stab grp} for the general method we use.
It will be convenient to use the presentation 
\[ \mathbb{F}_p[t_{i,j}: i\geq 1, 0\leq j\leq 1]/\left( t_{i,j}^p \mbox{\ \ for\ all\ } i,j\right) \]
for 
$E^0\left(\mathbb{F}_p[\strictAut({}_1\mathbb{G}^{\hat{\mathbb{Z}}_p\left[\sqrt{p}\right]}_{1/2})]^*\right)  \cong E^0\left( \mathbb{F}_p[t_1, t_2, \dots ]/\left( t_i^{p^2} - t_i\mbox{\ \ for\ all\ } i\right)\right)$, where $t_{i,j}$ is the image of $t_i^{p^j}$ in the associated graded.
The coproduct on $\mathbb{F}_p[t_{i,j}: i\geq 1, 0\leq j\leq 1]/\left( t_{i,j}^p \mbox{\ \ for\ all\ } i,j\right)$, inherited from that of $\mathbb{F}_p[\strictAut({}_1\mathbb{G}^{\hat{\mathbb{Z}}_p\left[\sqrt{p}\right]}_{1/2})]^*$, is given by
\[ \Delta(t_{i,j}) = \sum_{k=0}^i t_{k,j} \otimes t_{i-k, k+j}\]
for all $i < \floor{\frac{4p}{p-1}}$; reduce the $n=4$ case of 
Theorem~6.3.2 of~\cite{MR860042} 
modulo the ideal generated by $t_i^{p^2} - t_i$, for all $i$, to arrive at this formula.
\begin{description}
\item[$h_{10},h_{11}, \zeta_2$] There are no nonzero May differentials of any length on these classes, by the same computation as in the proof of Proposition~\ref{cohomology of ht 2 morava stab grp}.
\item[$h_{10}h_{30}, h_{11}h_{31}$] The class $h_{10}h_{30}$ is represented by 
the $2$-cocycle \[ t_{1,0}\otimes t_{3,0} - t_{1,0}\otimes t_{1,0}t_{2,0} - \frac{1}{2} t_{1,0}^2\otimes t_{2,0} + \frac{1}{2} t_{1,0}^2 \otimes t_{2,1} - \frac{1}{2} t_{1,0}^2 \otimes t_{1,0}t_{1,1} - \frac{1}{3} t_{1,0}^3\otimes t_{1,1}\]
in the cobar complex $C^{\bullet}\left(E^0\left(\mathbb{F}_p[\strictAut({}_1\mathbb{G}^{\hat{\mathbb{Z}}_p\left[ \sqrt{p}\right]}_{1/2})]^*\right)\right)$, 
which lifts to the $2$-cochain
\[ t_{1}\otimes t_{3} - t_{1}\otimes t_{1}t_{2} - \frac{1}{2} t_{1}^2\otimes t_{2} + \frac{1}{2} t_{1}^2 \otimes t_{2}^p - \frac{1}{2} t_{1}^2 \otimes t_{1}^{p+1} - \frac{1}{3} t_{1,0}^3\otimes t_{1}^p\]
in the cobar complex $C^{\bullet}\left(\mathbb{F}_p[\strictAut({}_1\mathbb{G}_{1/2})]^*\right)$.
Since this $2$-cochain is also a $2$-cocycle, all May differentials vanish on $h_{1,0}h_{3,0}$. The $C_2$-equivariance of the spectral sequence then tells us that all May 
differentials also vanish on $h_{11}h_{31}$.
\item[$h_{10}\eta_4, h_{11}\eta_4$] The only elements of internal degree $2(p-1)$ and cohomological degree $3$ are scalar multiples of $h_{10}\zeta_2\zeta_4$, but $h_{10}\zeta_2\zeta_4$ is of higher Ravenel degree than $h_{10}\eta_4$. Hence $h_{10}\eta_4$ cannot support a May differential of any length. By $C_2$-equivariance, the same is true of $h_{11}\eta_4$.
\item[$\eta_2e_{40}$] The only elements of internal degree $0$ and cohomological degree $4$ are $\mathbb{F}_p$-linear combinations of 
$\zeta_2\eta_2e_{40}, \zeta_4\eta_2e_{40}, \zeta_2\eta_4e_{40},$ and $\zeta_4\eta_4e_{40},$ but all four of these elements have higher Ravenel degree than $\eta_2e_{40}$, so again $\eta_2e_{40}$ cannot support a May differential of any length.
\item[$h_{10}\eta_2h_{30},h_{11}\eta_2h_{31},\eta_4e_{40}$] Similar degree considerations eliminate the possibility of nonzero May differentials on these classes.
\item[$\zeta_4$]
The class $\zeta_4$ is represented by 
the $1$-cocycle \[ t_{4,0} + t_{4,1} - t_{1,0}t_{3,1} - t_{1,1}t_{3,0} - \frac{1}{2}t_{2,0}^2 - \frac{1}{2}t_{2,1}^2 + t_{1,0}t_{1,1}t_{2,0} + t_{1,0}t_{1,1}t_{2,1}- \frac{1}{2}t_{1,0}^2t_{1,1}^2,\]
in the cobar complex $C^{\bullet}\left(E^0\left(\mathbb{F}_p[\strictAut({}_1\mathbb{G}^{\hat{\mathbb{Z}}_p\left[ \sqrt{p}\right]}_{1/2})]^*\right)\right)$, 
which lifts to the $1$-cochain
\[ t_{4} + t_{4}^p - t_{1}t_{3}^p - t_{1}^pt_{3} - \frac{1}{2}t_{2}^2 - \frac{1}{2}t_{2}^{2p} + t_{1}^{p+1}t_{2} + t_{1}^{p+1}t_{2}^p- \frac{1}{2}t_{1}^{2p+2},\]
in the cobar complex $C^{\bullet}\left(\mathbb{F}_p[\strictAut({}_1\mathbb{G}_{1/2})]^*\right)$.
Since this $1$-cochain is also a $1$-cocycle, all May differentials vanish on $\zeta_4$.
\end{description}
Now suppose that $q\geq 1$ is some integer and that we have already shown that $d_r$ vanishes on all classes, for all $r<q$. Then 
$d_r(\eta_2e_{40}\cdot \eta_4e_{40}) = 0$, i.e., $d_r$ vanishes on the duality class in the algebra $\mathcal{A}_{2,4}$.
For each element in that algebra, we have shown that $d_r$ vanishes on either that element, or on its Poincar\'{e} dual. Since $d_r$ also vanishes on the duality class, $d_r$ vanishes on all elements in that algebra.
Since $d_r$ also vanishes on $\zeta_2$ and $\zeta_4$, $d_r$ vanishes on all classes. By induction, the spectral sequence collapses with no nonzero differentials.
So 
$H^*(\strictAut({}_1\mathbb{G}^{\hat{\mathbb{Z}}_p\left[\sqrt{p}\right]}_{1/2}); \mathbb{F}_p) \cong \Cotor_{E^0\mathbb{F}_p[\strictAut({}_1\mathbb{G}^{\hat{\mathbb{Z}}_p\left[\sqrt{p}\right]}_{1/2})]^*}^{*,*,*}(\mathbb{F}_p,\mathbb{F}_p) \cong H^{*,*,*}(\mathcal{K}(2,4))$ as a graded $\mathbb{F}_p$-vector space. 
\end{proof}

\section{Topological consequences.}

It is well-known, e.g. from the Barsotti-Tate module generalization of the Dieudonn\'{e}-Manin classification of $p$-divisible groups over $\overline{k}$ (see~\cite{MR0157972}; also see~\cite{MR1393439} for a nice treatment of the theory of Barsotti-Tate modules), that the automorphism group scheme of a formal $A$-module of positive, finite height over a finite field is pro-\'{e}tale; in more down-to-earth terms, the Hopf algebra corepresenting the group scheme
$\Aut({}_{\omega}\mathbb{G}_{1/n}^{A}\otimes_{k}\overline{k})$ is the continuous $\overline{k}$-linear dual of the $\overline{k}$-linear group ring of some profinite group, namely, the automorphism group (honestly a group, not just a group scheme!) of ${}_{\omega}\mathbb{G}_{1/n}^{A}\otimes_{k}\overline{k}$. In this section we will cease to work with group schemes and we will simply write $\Aut({}_{\omega}\mathbb{G}_{1/n}^{A}\otimes_{k}\overline{k})$ for that profinite group.

The following is a generalization of a result in~\cite{formalmodules4}, and the argument is almost word-for-word the same:
\begin{theorem}\label{its a closed subgroup}
Let $K/\mathbb{Q}_p$ be a field extension of degree $d$. 
Let $A$ denote the ring of integers of $K$, and let $\pi$ denote a uniformizer for $A$ and $k$ the residue field of $A$. Let $q$ be the cardinality of $k$, and let $\omega$ denote a $\frac{q^{en}-1}{q^n-1}$th root of $\frac{\pi^e}{p}$ in $\overline{k}$.
Then $\Aut({}_{\omega}\mathbb{G}_{1/n}^{A}\otimes_{k}\overline{k})$ is a closed subgroup of the height $dn$ Morava stabilizer group $\Aut({}_1\mathbb{G}_{1/dn}^{\hat{\mathbb{Z}}_p}\otimes_{\mathbb{F}_p}\overline{k})$.
\end{theorem}
\begin{proof}
By Theorem~\ref{map induced in S(n)}, the underlying formal $\hat{\mathbb{Z}}_p$-module of ${}_{\omega}\mathbb{G}_{1/n}^{A}$ is ${}_1\mathbb{G}_{1/dn}^{\hat{\mathbb{Z}}_p}$.
Hence the automorphisms of ${}_{\omega}\mathbb{G}_{1/n}^{A}\otimes_{k}\overline{k}$ are the automorphisms of ${}_1\mathbb{G}_{1/dn}^{\hat{\mathbb{Z}}_p}\otimes_{\mathbb{F}_p}\overline{k}$ which commute with the complex multiplication by $A$, and hence
$\Aut({}_{\omega}\mathbb{G}_{1/n}^{A}\otimes_{k}\overline{k})$ is a subgroup of $\Aut({}_1\mathbb{G}_{1/dn}^{\hat{\mathbb{Z}}_p}\otimes_{\mathbb{F}_p}\overline{k})$.

Now let $G_a$ denote the automorphism group of the underlying formal $\hat{\mathbb{Z}}_p$-module $a$-bud of ${}_1\mathbb{G}_{1/dn}^{\hat{\mathbb{Z}}_p}\otimes_{\mathbb{F}_p}\overline{k}$, so that
$\Aut({}_1\mathbb{G}_{1/dn}^{\hat{\mathbb{Z}}_p}\otimes_{\mathbb{F}_p}\overline{k})$ is, as a profinite group, the limit of the sequence of finite groups $\dots \rightarrow G_3\rightarrow G_2 \rightarrow G_1$. Let $H_a$ denote the subgroup of 
$\Aut({}_1\mathbb{G}_{1/dn}^{\hat{\mathbb{Z}}_p}\otimes_{\mathbb{F}_p}\overline{k})$ consisting of those automorphisms whose underlying formal $\hat{\mathbb{Z}}_p$-module $a$-bud automorphism commutes with the complex multiplication by $A$, i.e., those
whose underlying formal $\hat{\mathbb{Z}}_p$-module $a$-bud automorphism is an automorphism of the underlying formal $A$-module $a$-bud of ${}_{\omega}\mathbb{G}_{1/n}^A$.
The index of $H_a$ in $\Aut({}_1\mathbb{G}_{1/dn}^{\hat{\mathbb{Z}}_p}\otimes_{\mathbb{F}_p}\overline{k})$ is at most the cardinality of $G_a$, hence is finite. Now we use the theorem of Nikolov-Segal, from~\cite{MR2276769}: every finite-index subgroup of a topologically finitely generated profinite group is an open subgroup.
The group $\Aut({}_1\mathbb{G}_{1/dn}^{\hat{\mathbb{Z}}_p}\otimes_{\mathbb{F}_p}\overline{k})$ is topologically finitely generated since it has the well-known presentation
\[ \Aut({}_1\mathbb{G}_{1/dn}^{\hat{\mathbb{Z}}_p}\otimes_{\mathbb{F}_p}\overline{k}) \cong \left( W(\overline{k})\langle S\rangle /(S^n = p, x^{\sigma} S = Sx)\right)^{\times},\]
where $W(\overline{k})\langle S\rangle$ is the Witt ring of the field $\overline{k}$ with a noncommuting polynomial indeterminate $S$ adjoined, and the relation $x^{\sigma} S = Sx$ is to hold for all $x\in W(\overline{k})$, where $x^{\sigma}$ is a lift of the Frobenius automorphism of $\overline{k}$ to $W(\overline{k})$, applied to $x$; since \[ W(\overline{k})\langle S\rangle /\left(S^n = p, x^{\sigma} S = Sx)\right)\] reduced modulo $S$ is a finite ring, $\Aut({}_1\mathbb{G}_{1/dn}^{\hat{\mathbb{Z}}_p}\otimes_{\mathbb{F}_p}\overline{k})$ has a dense finitely generated subgroup.

So $H_a$ is an open subgroup of $\Aut({}_1\mathbb{G}_{1/dn}^{\hat{\mathbb{Z}}_p}\otimes_{\mathbb{F}_p}\overline{k})$. Every open subgroup of a profinite group is also closed; consequently each $H_a$ is a closed subgroup of $\Aut({}_1\mathbb{G}_{1/dn}^{\hat{\mathbb{Z}}_p}\otimes_{\mathbb{F}_p}\overline{k})$, and consequently the intersection $\cap_a H_a$ is a closed subgroup of $\Aut({}_1\mathbb{G}_{1/dn}^{\hat{\mathbb{Z}}_p}\otimes_{\mathbb{F}_p}\overline{k})$. But $\cap_a H_a$ is the group of all formal power series which are automorphisms of ${}_1\mathbb{G}_{1/dn}^{\hat{\mathbb{Z}}_p}\otimes_{\mathbb{F}_p}\overline{k}$ and whose polynomial truncations, of any length, commute with the complex multiplication by $A$. Consequently $\cap_a H_a = \Aut({}_{\omega}\mathbb{G}_{1/n}^A\otimes_k\overline{k})$ is a closed subgroup of $\Aut({}_1\mathbb{G}_{1/dn}^{\hat{\mathbb{Z}}_p}\otimes_{\mathbb{F}_p}\overline{k})$.
\end{proof}

Since $\Aut({}_{\omega}\mathbb{G}_{1/n}^A\otimes_k\overline{k})$ is a closed subgroup of the height $dn$ Morava stabilizer group $\Aut({}_1\mathbb{G}_{1/dn}^{\hat{\mathbb{Z}}_p}\otimes_{\mathbb{F}_p}\overline{k})$, we can use the methods of~\cite{MR2030586} to construct and compute the homotopy fixed-point spectra 
\[ E_4^{h\Aut({}_{\omega}\mathbb{G}_{1/n}^{A}\otimes_{k}\overline{k})} \mbox{\ \ \ and\ \ \ } E_4^{h\Aut({}_{\omega}\mathbb{G}_{1/n}^{A}\otimes_{k}\overline{k})\rtimes \Gal(\overline{k}/k)}.\]
The homotopy fixed-point spectrum 
$E_4^{h\Aut({}_{1}\mathbb{G}_{1/dn}^{\hat{\mathbb{Z}}_p}\otimes_{\mathbb{F}_p}\overline{k})\rtimes \Gal(\overline{k}/\mathbb{F}_p)}\simeq L_{K(4)}S$ 
has a natural map to
$E_4^{h\Aut({}_{\omega}\mathbb{G}_{1/n}^{A}\otimes_{k}\overline{k})\rtimes \Gal(\overline{k}/\mathbb{F}_p)}$, but this map is far from being an equivalence; still, very few computations of homotopy groups of $K(4)$-local spectra exist in the literature, so Theorem~\ref{topological computation} is perhaps of some interest.
\begin{theorem}\label{topological computation}
Let $p$ be a prime number such that the Smith-Toda complex $V(3)$ exists, i.e., $p>5$. Then the $V(3)$-homotopy groups of
$E_4^{h\Aut({}_{1}\mathbb{G}_{1/2}^{\hat{\mathbb{Z}}_p[\sqrt{p}]}\otimes_{\mathbb{F}_p}\overline{\mathbb{F}}_p)\rtimes \Gal(\overline{k}/\mathbb{F}_p)}$ are isomorphic to
\[ \mathcal{A}_{2,4} \otimes_{\mathbb{F}_p} \Lambda(\zeta_2,\zeta_4)\otimes_{\mathbb{F}_p} \mathbb{F}_p[v^{\pm 1}],\]
where $v^{p^2+1} = v_4$,
where $\mathcal{A}_{2,4}$ is as in Proposition~\ref{coh of E 4 4 0}, and the topological degrees and $E_4$-Adams filtrations are as follows:
\[
\begin{array}{llllll}
\mbox{Htpy.\ class}          & \mbox{Top.\ degree} & E_4-\mbox{Adams\ filt.} \\
\hline \\
1                                 & 0              & 0 \\
h_{10}                            & 2p-3           & 1 \\
h_{11}                            & 2p^2-2p-1       & 1 \\
h_{10}h_{30}                       & 4p-6          & 2 \\
h_{11}h_{31}                       & 4p^2-4p-2      & 2 \\
h_{10}\eta_4-\eta_2h_{30}          & 2p-4           & 2 \\
h_{11}\eta_4-\eta_2h_{31}          & 2p^2-2p-2      & 2 \\
\eta_2e_{40}                      & -3             & 3 \\
h_{10}\eta_2h_{30}                 & 4p-7          & 3 \\
h_{11}\eta_2h_{31}                 & 4p^2-4p-3     & 3 \\
h_{10}h_{30}\eta_4                 & 4p-7          & 3 \\
h_{11}h_{31}\eta_4                 & 4p^2-4p-3     & 3 \\
\eta_4e_{40}+4\eta_2h_{30}h_{31}   & -3             & 3 \\
h_{10}\eta_2h_{30}h_{31}           & 2p-6           & 4 \\
h_{11}\eta_2h_{30}h_{31}           & 2p^2-2p-4       & 4 \\
h_{10}\eta_2h_{30}\eta_4           & 4p-8           & 4 \\
h_{11}\eta_2h_{31}\eta_4           & 4p^2-4p-4      & 4 \\
h_{10}\eta_2h_{30}h_{31}\eta_4      & 2p-7          & 5 \\
h_{11}\eta_2h_{30}h_{31}\eta_4      & 2p^2-2p-5     & 5 \\
h_{10}h_{11}\eta_2h_{30}h_{31}\eta_4 & -6            & 6 \\
\hline \\
\zeta_2                          & -1                   & 1 \\
\zeta_4                          & -1                   & 1 \\
\hline \\
v                                & 2p^2-2           & 0 
    \end{array} \]
\end{theorem}
\begin{proof}
See~\cite{MR1333942} and ~\cite{MR2030586} for the equivalence
\[ L_{K(n)}S \simeq E_n^{h\Aut({}_1\mathbb{G}_{1/n}^{\hat{\mathbb{Z}}_p}\otimes_{\mathbb{F}_p}\overline{\mathbb{F}}_{p})\rtimes\Gal(\overline{\mathbb{F}}_p/\mathbb{F}_p)} .\]
Since $V(3)$ is $E(3)$-acyclic, $L_{K(4)}V(3)$ is weakly equivalent to $L_{E(4)}V(3)$, so $L_{K(4)}V(3) \simeq L_{E(4)}V(3) \simeq V(3) \smash L_{E(4)}S$ since $E(4)$-localization is smashing; see~\cite{MR1192553} for the proof of Ravenel's smashing conjecture.
Since $V(3)$ is finite, $(E_4\smash V(3))^{hG} \simeq E_4^{hG} \smash V(3)$, and now we use the $X = V(3)$ case of the conditionally convergent descent spectral sequence (see e.g.~4.6 of~\cite{MR2645058}, or~\cite{MR2030586})
\begin{align*} 
 E_2^{s,t} \cong H^s_c(G; (E_n)_t(X)) & \Rightarrow \pi_{t-s}((E_n\smash X)^{hG}) \\
  d_r: E_r^{s,t} & \rightarrow E_r^{s+r,t+r-1}.\end{align*}
The agreement of this spectral sequence with the $K(4)$-local $E_4$-Adams spectral sequence is given by Proposition~6.6 of~\cite{MR2030586}.

In the case $n=4$ and $X =V(3)$, we have $(E_4)_*\cong W(\overline{\mathbb{F}}_{p})[[u_1,u_2,u_3]][u^{\pm 1}]$ with $v_i$ acting by $u_iu^{1-p^i}$ for $i=1,2,3$, and consequently $(E_4)_*(V(3)) \cong \overline{\mathbb{F}}_p[u^{\pm 1}]$. One needs to know the action of $\Aut({}_1\mathbb{G}_{1/2}^{\hat{\mathbb{Z}}_p[\sqrt{p}]}\otimes_{\mathbb{F}_p}\overline{\mathbb{F}}_{p})\rtimes \Gal(\overline{\mathbb{F}}_p/\mathbb{F}_p)$ on $\overline{\mathbb{F}}_{p}[u^{\pm 1}]$ to compute the $E_2$-term of the spectral sequence; but $\Aut({}_1\mathbb{G}_{1/2}^{\hat{\mathbb{Z}}_p[\sqrt{p}}\otimes_{\mathbb{F}_p}\overline{\mathbb{F}}_{p})$ has the finite-index pro-$p$-subgroup 
$\strictAut({}_1\mathbb{G}_{1/2}^{\hat{\mathbb{Z}}_p[\sqrt{p}}\otimes_{\mathbb{F}_p}\overline{\mathbb{F}}_{p})$. As a pro-$p$-group admits no nontrivial continuous action on a one-dimensional vector space over a field of characteristic $p$, we only need to know the action of $\Aut({}_1\mathbb{G}_{1/2}^{\hat{\mathbb{Z}}_p[\sqrt{p}]}\otimes_{\mathbb{F}_p}\overline{\mathbb{F}}_p)/\strictAut({}_1\mathbb{G}_{1/2}^{\hat{\mathbb{Z}}_p[\sqrt{p}]}\otimes_{\mathbb{F}_p}\overline{\mathbb{F}}_p)\cong \mathbb{F}_{p^2}^{\times}$ on $\overline{\mathbb{F}}_{p}[u^{\pm 1}]$; 
from section~1 of~\cite{MR1333942} we get that an element $x\in \mathbb{F}_{p^2}^{\times}$ acts on 
$\overline{\mathbb{F}}_{p}\{u^j\}$ by multiplication by $x^j$. 
Consequently
the (collapsing at $E_2$) Lyndon-Hochschild-Serre spectral sequence of the extension
\[ 1 \rightarrow \strictAut({}_1\mathbb{G}_{1/2}^{\hat{\mathbb{Z}}_p[\sqrt{p}]}\otimes_{\mathbb{F}_p}\overline{\mathbb{F}}_p) \rightarrow \Aut({}_1\mathbb{G}_{1/2}^{\hat{\mathbb{Z}}_p[\sqrt{p}]}\otimes_{\mathbb{F}_p}\overline{\mathbb{F}}_p) \rightarrow \mathbb{F}_{p^2}^{\times} \rightarrow 1\]
gives us that $H^*_c(\Aut({}_1\mathbb{G}_{1/2}^{\hat{\mathbb{Z}}_p[\sqrt{p}]}\otimes_{\mathbb{F}_p}\overline{\mathbb{F}}_p); V(3)_t(E_4^{h\Aut({}_1\mathbb{G}_{1/2}^{\hat{\mathbb{Z}}_p[\sqrt{p}]}\otimes_{\mathbb{F}_p}\overline{\mathbb{F}}_p)}))$ vanishes if $t$ is not divisible by $2(p^2-1)$, and is given by Proposition~\ref{coh of ht 2 fm} if $t$ is divisible by $2(p^2-1)$.
So there is a horizontal vanishing line of finite height already at the $E_2$-page of the spectral sequence, hence the spectral sequence converges strongly.

More specifically, the cohomology computed in Proposition~\ref{coh of ht 2 fm} is a $\Gal(\overline{\mathbb{F}}_p/\mathbb{F}_p)$-form of 
\[ H^*_c(\Aut({}_1\mathbb{G}_{1/2}^{\hat{\mathbb{Z}}_p[\sqrt{p}]}\otimes_{\mathbb{F}_p}\overline{\mathbb{F}}_p); V(3)_{2(p^2-1)j}(E_4^{h\Aut({}_1\mathbb{G}_{1/2}^{\hat{\mathbb{Z}}_p[\sqrt{p}]}\otimes_{\mathbb{F}_p}\overline{\mathbb{F}}_p)})),\] and
\[ H^*_c(\Aut({}_1\mathbb{G}_{1/2}^{\hat{\mathbb{Z}}_p[\sqrt{p}]}\otimes_{\mathbb{F}_p}\overline{\mathbb{F}}_p); V(3)_{2(p^2-1)j}(E_4^{h\Aut({}_1\mathbb{G}_{1/2}^{\hat{\mathbb{Z}}_p[\sqrt{p}]}\otimes_{\mathbb{F}_p}\overline{\mathbb{F}}_p)}))^{\Gal(\overline{\mathbb{F}}_p/\mathbb{F}_p)}\] is also a $\Gal(\overline{\mathbb{F}}_p/\mathbb{F}_p)$-form of 
\[ H^*_c(\Aut({}_1\mathbb{G}_{1/2}^{\hat{\mathbb{Z}}_p[\sqrt{p}]}\otimes_{\mathbb{F}_p}\overline{\mathbb{F}}_p); V(3)_{2(p^2-1)j}(E_4^{h\Aut({}_1\mathbb{G}_{1/2}^{\hat{\mathbb{Z}}_p[\sqrt{p}]}\otimes_{\mathbb{F}_p}\overline{\mathbb{F}}_p)})).\] Since the nonabelian Galois cohomology group $H^1(\Gal(\overline{\mathbb{F}}_{p}/\mathbb{F}_p); GL_n(\overline{\mathbb{F}}_{p}))$ classifying $\Gal(\overline{\mathbb{F}}_{p}/\mathbb{F}_p)$-forms of $n$-dimensional $\overline{\mathbb{F}}_{p}$-vector spaces vanishes (this is a well-known generalization of Hilbert's Theorem 90), the invariants of the $\Gal(\overline{\mathbb{F}}_{p}/\mathbb{F}_p)$-action on \[ H^*_c(\Aut({}_1\mathbb{G}_{1/2}^{\hat{\mathbb{Z}}_p[\sqrt{p}]}\otimes_{\mathbb{F}_p}\overline{\mathbb{F}}_p); V(3)_{2(p^2-1)j}(E_4^{h\Aut({}_1\mathbb{G}_{1/2}^{\hat{\mathbb{Z}}_p[\sqrt{p}]}\otimes_{\mathbb{F}_p}\overline{\mathbb{F}}_p)}))\] agree, up to isomorphism of graded $\mathbb{F}_p$-vector spaces, with the results of Proposition~\ref{coh of ht 2 fm} (this Galois descent argument was suggested to me by T. Lawson). There is no room for differentials in the descent spectral sequence, so $E_2\cong E_{\infty}$ in the spectral sequence.
\end{proof}

\bibliography{/home/asalch/texmf/tex/salch}{}
\bibliographystyle{plain}
\end{document}